\documentclass[a4paper]{amsart}

\usepackage{stmaryrd}
\usepackage{amsfonts,amssymb, amsmath,amsthm,mathtools,tikz-cd}
\usepackage{hyperref}
\usepackage{enumitem}

\theoremstyle{plain}
\newtheorem{lemma}{Lemma}[section]
\newtheorem{thm}[lemma]{Theorem}
\newtheorem{prop}[lemma]{Proposition}
\newtheorem{cor}[lemma]{Corollary}

\theoremstyle{definition}
\newtheorem{defi}[lemma]{Definition}
\newtheorem{example}[lemma]{Example}
\newtheorem{rem}[lemma]{Remark}

\numberwithin{equation}{section}

\newcommand{\enum}{\rm{(\roman*)}}

\newcommand{\N}{\ensuremath {\mathbb{N}}}
\newcommand{\Z} {\ensuremath {\mathbb{Z}}}
\newcommand{\R} {\ensuremath {\mathbb{R}}}

\newcommand{\calA}{\ensuremath\mathcal{A}}
\newcommand{\calB}{\ensuremath\mathcal{B}}
\newcommand{\calC}{\ensuremath\mathcal{C}}
\newcommand{\calL}{\ensuremath\mathcal{L}}
\newcommand{\calH}{\ensuremath\mathcal{H}}

\newcommand{\frakX} {\ensuremath {\mathfrak{X}}}
\newcommand{\BAc}{\ensuremath\mathsf{BAc}}

\newcommand{\op}{\mathrm{op}}
\newcommand{\sn}{\mathrm{sn}}
\newcommand{\gp}{\mathrm{gp}}
\newcommand{\un}{\mathrm{un}}

\def\rho{\varrho}
\def\phi{\varphi}

\DeclareMathOperator{\Set}{Set}
\DeclareMathOperator{\Top}{Top}
\DeclareMathOperator{\Mod}{Mod}
\DeclareMathOperator{\SSet}{SSet}
\DeclareMathOperator{\Grpd}{Grpd}
\DeclareMathOperator{\Cat}{Cat}

\DeclareMathOperator{\Hom}{Hom}
\DeclareMathOperator{\Sing}{Sing}
\DeclareMathOperator{\Tot}{Tot}

\DeclareMathOperator{\UBC}{UBC}
\DeclareMathOperator{\BBC}{BBC}

\DeclareMathOperator{\tr}{tr}
\DeclareMathOperator{\comp}{comp}
\DeclareMathOperator{\pr}{pr}
\DeclareMathOperator{\im}{im}
\DeclareMathOperator{\id}{id}

\newcommand{\into}{\ensuremath\hookrightarrow}
\newcommand{\onto}{\ensuremath\twoheadrightarrow}

\newcommand{\firstE}{{}^{I}_{}\! E}
\newcommand{\secondE}{{}^{I\! I}_{}\! E}
\newcommand{\firstd}{{}^{I}_{}\! d}
\newcommand{\secondd}{{}^{I\! I}_{}\! d}
\newcommand{\bb}{{\bullet,\bullet}}

\makeatletter
\def\bprod#1{%
  \begingroup%
   \if@display%
      \def\@tempa{\sideset{}{^b}\prod_{#1}}%
   \else%
      \def\@tempa{\prod_{#1}^b}%
   \fi%
   \@tempa\endgroup}
\makeatother

\makeatletter
\def\lsum#1{%
  \begingroup%
   \if@display%
      \def\@tempa{\sideset{}{^{\ell^1}}\bigoplus_{#1}}%
   \else%
      \def\@tempa{\bigoplus_{#1}^{\ell^1}}%
   \fi%
   \@tempa\endgroup}
\makeatother

\title[The Serre spectral sequence in bounded cohomology]{The Serre spectral sequence in \\ bounded cohomology}

\author[K.~Li]{Kevin Li}
\address{
Fakult\"{a}t f\"{u}r Mathematik, Universit\"{a}t Regensburg, 93040 Regensburg, Germany}
\email{kevin.li@ur.de}

\author[M.~Moraschini]{Marco Moraschini}
\address{
Dipartimento di Matematica, Universit\`{a} di Bologna, 40126 Bologna, Italy}
\email{marco.moraschini2@unibo.it}

\author[G.~Raptis]{George Raptis}
\address{
Department of Mathematics, Aristotle University of Thessaloniki, 54124 Thessaloniki, Greece}
\email{raptisg@math.auth.gr}

\date{\today}

\begin{document}

\begin{abstract}
We construct the analogue of the Serre spectral sequence for the bounded cohomology of simplicial sets with seminormed local coefficients. As applications, we obtain a (non-isometric) generalization of Gromov's mapping theorem and some partial results on the 
simplicial volume of manifold bundles. 
  \end{abstract}

\maketitle

\section{Introduction}

The classical Serre spectral sequence (also called Leray--Serre spectral sequence) \cite{serre1951homologie} is a standard tool in algebraic topology that has been extensively used for many different (co)homological computations. For a Serre fibration $F \to E \to B$ and a local coefficient system $\calA$ on $E$, the Serre spectral sequence is a spectral sequence of the form
$$E^{p,q}_2 = H^p(B; \calH^q(F; \calA)) \Rightarrow H^{p+q}(E; \calA)$$
that assembles the (co)homological information of the fibration. It relates the (co)homology of the fiber, base, and total space of the fibration, and it has been used in numerous applications for obtaining new (co)homological computations from old ones. A standard construction of the Serre spectral sequence -- but not the only one possible -- is obtained from the skeletal filtration of the base $B$, which may be assumed to be a CW-complex, and crucially makes use of excision for the identification of the $E_2$-page.

Bounded cohomology of topological spaces (or groups) is an interesting variant of singular cohomology with deep connections to group theory and differential geometry \cite{vbc}. Bounded cohomology does not satisfy excision which makes its computation difficult in general -- even for tame spaces like the wedge of two circles, the bounded cohomology has been computed only in degrees~$\leq 3$~\cite{brooks1980some, grigorchuk1993some, soma1997bounded}. Moreover, the interaction with delicate functional-analytic matters also makes the development of standard homological algebra methods in bounded cohomology generally more challenging.  

The purpose of this paper is to construct the analogue of the Serre spectral sequence for bounded cohomology and discuss some applications. For several reasons, it will be both convenient and important to work with simplicial sets rather than topological spaces. Firstly, both bounded cohomology of topological spaces and of groups are special cases of bounded cohomology of simplicial sets, via the singular set functor and the nerve functor, respectively. So the context of simplicial sets contains both special cases of interest while also avoiding some of their known technical issues. 

Secondly, whereas topological spaces and simplicial sets can essentially be used interchangeably in algebraic topology, this is not completely true from the viewpoint of bounded cohomology,  simply because weak equivalences of simplicial sets are not bounded cohomology equivalences in general. This means that working with simplicial sets (not necessarily Kan complexes) extends the context of application of the spectral sequence. The use of simplicial methods in bounded cohomology has proved very successful in recent years~\cite{ivanov2020simplicial, Monod-Nariman23,Kastenholz-Sroka23}. 

Moreover, it is also important to work here with general local coefficients in seminormed modules over a normed ring~$R$ (rather than normed, Banach or dual Banach modules). For a simplicial set $X$, a system of local coefficients on $X$ is given by a functor $\calA\colon \pi(X)\to \Mod^\sn_R$ from the fundamental groupoid~$\pi(X)$ of~$X$ to the category of seminormed $R$-modules. The bounded cohomology~$H^*_b(X;\calA)$ is defined analogously to ordinary simplicial cohomology, but using only bounded cochains. Seminormed local coefficients are essentially required for the crucial identification of the $E_2$-page of the spectral sequence, since bounded cohomology groups are only seminormed modules in general. Moreover, if the values of~$\calA$ carry the trivial seminorm, then bounded cohomology recovers ordinary cohomology. Hence the Serre spectral sequence in bounded cohomology may be regarded as a generalization of the classical Serre spectral sequence to a seminormed context.
Our main result is the following:

\begin{thm}[Theorem~\ref{thm:sss}]
\label{thm:sss intro}
	Let~$f\colon X\to Y$ be a Kan fibration of simplicial sets and let~$\calA$ be a seminormed local coefficient system on~$X$.
	Then there is a first-quadrant spectral sequence~$(E^{\bullet,\bullet}_r)_r$ converging to $H^*_b(X;\calA)$.
    We have an isomorphism of seminormed $R$-modules
    \[
        E^{p,q}_2\cong H^p_b(Y;\calH^q_b(F;\calA))
    \]
    in each of the following cases:
    \begin{enumerate}[label=\enum]
        \item $Y$ has only finitely many $n$-simplices in every dimension $n \leq p+1$; 
        \item $q=0$;
        \item $Y$ is connected and for some 0-simplex~$y\in Y_0$ the bounded $R$-cochain complex~$C^\bullet_b(F_y;\calA_y)$ of the fiber~$F_y$ satisfies the uniform boundary condition in degree~$q$.
    \end{enumerate}
\end{thm}
Unlike in ordinary cohomology, additional assumptions are needed in bounded cohomology in order to identify the terms on the $E_2$-page as anticipated. Indeed, while cohomology commutes with arbitrary products, it is not true in general that (bounded) cohomology commutes with \emph{bounded} products of seminormed cochain complexes. However, in each of the cases~(i), (ii), and~(iii), cohomology in fact commutes with the bounded product of bounded cochain complexes of all fibers.
Case~(i) showcases an advantage of the framework of simplicial sets over the setting of topological spaces.
In case~(iii), provided that~$\calA$ takes values in Banach $\R$-modules, the cochain complex $C^\bullet_b(F_y;\calA_y)$ satisfies the uniform boundary condition in degree~$q$ if and only if~$H^q_b(F_y;\calA_y)$ is Banach \cite{Matsu-Mor} (Corollary~\ref{cor:matsumoto-morita}).
If the values of~$\calA$ carry the trivial seminorm, then the uniform boundary condition is trivially satisfied and we recover the Serre spectral sequence in ordinary cohomology.
Functoriality of the construction yields a \emph{comparison map} between the corresponding Serre spectral sequences (Proposition~\ref{prop:comparison ss}).

The Serre spectral sequence of Theorem~\ref{thm:sss intro} specializes to a Serre spectral sequence in bounded cohomology for Serre fibrations of topological spaces (Theorem~\ref{thm:Serre fibrations}) and to the Lyndon--Hochschild--Serre spectral sequence in bounded cohomology for group extensions~\cite{Noskov92,Burger_Monod_2002} (Theorem~\ref{thm:LHS}).

Interestingly, unlike in ordinary cohomology, the spectral sequence of Theorem~\ref{thm:sss intro} cannot be obtained from the skeletal filtration of the base. This filtration yields another and less interesting spectral sequence in general -- note that the induced filtration of $X$ very soon consists of $\pi_1$-isomorphisms which are bounded cohomology equivalences. The proof of Theorem \ref{thm:sss intro} follows a different elegant construction of the Serre spectral sequence due to Dress~\cite{dress}. This does not require excision and seems closer to 
the approach (cf.~\emph{Leray spectral sequence}) that yields the Serre spectral sequence as a hypercohomology spectral sequence.

In analogy with Theorem~\ref{thm:sss intro}, we construct similarly the analogous Serre spectral sequence in $\ell^1$-homology (Section~\ref{sec:l1-homology}).

\medskip
We give several applications of the Serre spectral sequence in bounded cohomology.
Since $n$-connected Kan complexes have trivial bounded cohomology and satisfy the uniform boundary condition in degrees~$\le n$ (Corollary~\ref{cor:n-connected bc}), we deduce from Theorem~\ref{thm:sss intro} that $(n+1)$-equivalences (which have $n$-connected fibers) induce isomorphisms in bounded cohomology in degrees~$\le n$ (Corollary~\ref{cor:n-invariance bc}). 

Similarly, if the higher homotopy groups of a Kan complex are finite, amenable, or boundedly acyclic, then its bounded cohomology is trivial for all Banach coefficients, all dual Banach coefficients, or trivial coefficients in~$\R$, respectively (Lemma~\ref{lem:acyclic fiber}). In addition, for Banach coefficients, the vanishing of bounded cohomology implies the uniform boundary condition.
Hence Theorem~\ref{thm:sss intro} shows that Kan fibrations with fibers as above induce isomorphisms in bounded cohomology:

\begin{thm}[Theorem~\ref{thm:mapping thm class}]
\label{thm:mapping thm}
	Let~$f\colon X\to Y$ be a Kan fibration between connected simplicial sets with connected fiber~$F$.
	Let~$\calA$ be a seminormed local coefficient system on~$X$ taking values in Banach $\R$-modules.
	Suppose that we are in one of the following situations:
	\begin{enumerate}[label=\enum]
		\item\label{item:mapping thm finite} $\pi_k(F)$ is finite for all~$k\ge 1$;
		\item\label{item:mapping thm amenable} $\pi_1(F)$ is amenable and~$\calA$ takes values in dual Banach modules;
		\item\label{item:mapping thm bac} $H^i_b(\pi_1(F);\R)=0$ for all~$i \geq 1$ and~$\calA$ restricts on $F$ to a constant coefficient system at a dual Banach module. 
	\end{enumerate}
	Then the canonical map
	\[
		 H^*_b(Y;\calH^0_b(F;\calA))\to H^*_b(X;\calA)
	\]
	is an isomorphism (of $\R$-modules) in all degrees.
\end{thm}

Theorem~\ref{thm:mapping thm}~\ref{item:mapping thm amenable} recovers (a non-isometric version of) Gromov's mapping theorem, which is a fundamental result in the theory of bounded cohomology:

\begin{thm}[\cite{vbc}]
\label{thm:Gromov mapping thm}
    Let $f\colon E\to B$ be a map between path-connected topological spaces such that the induced map $\pi_1(f)\colon \pi_1(E)\to \pi_1(B)$ is surjective with amenable kernel.
    Then the induced map
    \[
        H^*_b(f)\colon H^*_b(B;\R)\to H^*_b(E;\R)
    \]
    is an isometric isomorphism in all degrees.
\end{thm}

Theorem~\ref{thm:Gromov mapping thm} was proved by Gromov~\cite{vbc} using multicomplexes~\cite{frigeriomoraschini} and by Ivanov \cite{Ivanov85,Ivanov17} using relatively injective resolutions. More recently, Ivanov also provided a proof of Theorem~\ref{thm:mapping thm}~\ref{item:mapping thm amenable} for simplicial sets~\cite[Theorem~7.6]{ivanov2020simplicial}.
The advantage of our proof using the Serre spectral sequence is its simplicity and generality.
However, our construction does not provide \emph{isometric} isomorphisms. Theorem~\ref{thm:mapping thm}~\ref{item:mapping thm finite} and~\ref{item:mapping thm bac} provide (partial) simplicial versions of results by the second and third authors~\cite[Theorem~B and Theorem~C]{Moraschini-Raptis} on the bounded cohomology of discrete groups and of topological spaces.

In another direction, we hope that the Serre spectral sequence will have applications to the computation of the simplicial volume of manifold bundles. We discuss some 
partial results in Section~\ref{sec:sv}. Specifically, under appropriate conditions, we show that the (non-)vanishing of the simplicial volume of the fiber and the base manifolds implies the (non-)vanishing of the simplicial volume of the total space (Proposition~\ref{prop:sv vanishing} and Proposition~\ref{prop:sv positive}).

\medskip
\noindent\textbf{Organization of the paper.} 
In Section~\ref{sec:bc:sset:local:coeff} we define the bounded cohomology of simplicial sets with seminored local coefficients. We also discuss the compatibility of bounded cohomology with bounded products and the role of the uniform boundary condition.
In Section~\ref{sec:weak invariance bc} we prove the homotopy invariance of bounded cohomology and of the uniform boundary condition.
Section~\ref{sec:sss} is devoted to the construction of the Serre spectral sequence in bounded cohomology (Theorem~\ref{thm:sss intro}).
The Serre spectral sequence in $\ell^1$-homology is outlined in Section~\ref{sec:l1-homology}.
Section~\ref{sec:applications} consists of several applications of Theorem~\ref{thm:sss intro}, including the proof of Theorem~\ref{thm:mapping thm} and some partial results concerning the simplicial volume of manifold bundles.

\section{Bounded cohomology of simplicial sets with local coefficients}\label{sec:bc:sset:local:coeff}

\subsection{Preliminaries on simplicial sets}\label{subsec:sset} We recall some basic notions of simplicial homotopy theory. The interested reader should consult some of the standard references (e.g., \cite{GZ_fractions, hovey2007model, goerss2009simplicial}) for more details. 

We denote by $\Delta$ the \emph{simplex category}, that is, the category whose objects are the finite ordinals $
[n] \coloneqq \{0 < \cdots < n\}
$
and the morphisms are order-preserving functions $[m] \to [n]$.
A \emph{simplicial set} is a functor 
$
X \colon \Delta^{\textup{op}} \to \Set
$
from the opposite of the simplex category to the category of sets. A \emph{(simplicial) map} $f \colon X \to Y$ between simplicial sets is a natural transformation of functors. We denote by $\SSet$ the \emph{category of simplicial sets}. Limits and colimits in $\SSet$ are defined pointwise by limits and colimits of sets. 

Given a simplicial set $X$, we denote by $\partial_i^n \colon X([n]) \to X([n-1])$ the $i$-\emph{th face map} which is induced by the inclusion $\delta^n_i \colon [n-1] \to [n]$ that skips the $i$-th entry in the image.
The \emph{standard $n$-simplex} $\Delta^n$ is the simplicial set represented by $[n] \in \Delta$, that is, the functor $\textup{Hom}_{\Delta}(-,[n]) \colon \Delta^{\op} \to \Set$. By the Yoneda lemma, the \emph{set of $n$-simplices} $X_n \coloneqq X([n])$ of a simplicial set $X$ is naturally isomorphic to the set of maps $\Delta^n \to X$. In particular, the Yoneda embedding $j \colon \Delta \to \SSet$, $[n] \mapsto \Delta^n$, is full and faithful.

For two simplicial sets $X$ and $Y$, the (simplicial) \emph{mapping space} $Y^X \colon \Delta^{\op} \to \Set$ is the simplicial set whose $n$-simplices are given by maps $X \times \Delta^n \to Y$. These mapping space objects turn the category of simplicial sets into a cartesian closed category. 

\begin{example}[Nerve of a category]\label{defi:nerve}
Let $C$ be a small category (i.e., a category such that the collection of morphisms forms a set instead of a proper class). 
The \emph{nerve}~$N(C)$ of $C$ is a simplicial set whose $n$-simplices are given by functors $[n] \to C$, where the poset~$[n]$ is viewed as a category. 
Hence an $n$-simplex of~$N(C)$ is a sequence of $n$ composable morphisms in $C$. The simplicial operators are defined by precomposition of functors. This construction extends to a functor 
$$N \colon \Cat \rightarrow \SSet, \ \quad \ C \mapsto (N(C) \colon \Delta^{\mathrm{op}} \to \mathrm{Set}, \ [n] \mapsto \Hom_{\Cat}([n], C)),$$
where $\Cat$ denotes the \emph{category of small categories and functors}. The nerve functor is full and faithful.
\end{example}

\begin{example}[Simplex category of a simplicial set]
\label{defn:simplex category}
Let $X$ be a simplicial set. The \emph{simplex category} $\Delta \shortdownarrow X$ of $X$ is the category whose objects are the maps $\sigma \colon \Delta^n \to X$, $n \geq 0$, that is, the simplices of~$X$, and a morphism $(\sigma \colon \Delta^n \to X) \to (\tau \colon \Delta^m \to X)$ is given by a map $u \colon [n] \to [m]$ in~$\Delta$ such that $\tau\circ j(u)=\sigma$.   
\end{example}

Every simplicial set $X$ is the colimit of the canonical diagram: 
\[
(\Delta \shortdownarrow X) \to \Delta \xrightarrow{j} \SSet, \ \ (\sigma \colon \Delta^n \to X) \mapsto \Delta^n,
\]
where the first functor is the obvious forgetful functor. 
This implies that the Yoneda embedding $j \colon \Delta \to \SSet$ has the following universal property: for every category~$C$ with small colimits and a functor (cosimplicial object) $\alpha \colon \Delta \to C$, there is a colimit-preserving functor, unique up to unique isomorphism,
$$F_\alpha \colon \SSet \rightarrow C$$
with $F_\alpha \circ j = \alpha$. Specifically, $F_\alpha$ is given on objects by 
$$F_\alpha(X) = \mathrm{colim}(\Delta \shortdownarrow X \to \Delta \xrightarrow{\alpha} C).$$
The functor~$F_\alpha$ admits a right adjoint $G_\alpha \colon C \to \SSet$ which (by adjunction) is defined on objects by $G_\alpha(Y)_n = \hom_{C}(\alpha([n]), Y)$. 

Applying this construction to the usual cosimplicial object in topological spaces,
$$\Delta \to \Top, \ \ [n] \mapsto |\Delta^n|,$$ 
given by the standard topological simplices, we obtain the usual adjunction 
$$|\cdot| \colon \SSet \rightleftarrows \Top \colon \Sing$$
between the geometric realization and the singular set functors. 
The geometric realization functor sends a simplicial set $X$ to a topological space $|X|$ which has the structure of a CW-complex with one $n$-cell for each non-degenerate $n$-simplex of~$X$.
The singular set $\Sing(X) \colon \Delta^{\textup{op}} \to \Set$ of a topological space~$X$ is the simplicial set whose $n$-simplices are given by the continuous maps $|\Delta^n| \to X$.

On the other hand, the (tautological) cosimplicial object in $\Cat$, 
$$\Delta \to \Cat, \ \ [n] \mapsto [n],$$
yields the adjunction
$$\tau \colon \SSet \rightleftarrows \Cat \colon N$$
where the left adjoint functor $\tau$ is known as the \emph{fundamental category}.

For every $k \leq n$ we define the $k$-\emph{th horn of} $\Delta^n$ as the subobject $\Lambda^n_k$ of $\Delta^n$ obtained by removing the $n$-simplex $\id_{[n]} \in (\Delta^n)_n$  and its $k$-th face in $(\Delta^n)_{n-1}$. We denote by $i\colon \Lambda^n_k\hookrightarrow \Delta^n$ the inclusion. A map $f \colon X \to Y$ of simplicial sets is a \emph{Kan fibration} if for every $n \geq 1$, $k\in \{0,\ldots,n\}$, and every commutative square in $\SSet$ 
    \[\begin{tikzcd}
        \Lambda^n_k\ar{rr}{s}\ar{d}[swap]{i}
        && X\ar{d}{f}
        \\
        \Delta^n\ar{rr}[swap]{\tau}\ar[dashed]{urr}{\widetilde{\tau}}
        && Y
    \end{tikzcd}\]
there is a map $\widetilde{\tau} \colon \Delta^n \to X$ making the diagram commutative. A simplicial set $X$ is a \emph{Kan complex} if the map $X \to \Delta^0$ is a Kan fibration. For every Serre fibration $p \colon E \to B$ of topological spaces, the induced map $\Sing(p) \colon \Sing(E) \to \Sing(B)$ is a Kan fibration~\cite[p.~11]{goerss2009simplicial}. In particular, the singular set of a topological space is a Kan complex. \ Moreover, for every Kan fibration $f \colon X \to Y$ of simplicial sets, the map $|f| \colon |X| \to |Y|$ is Serre fibration~\cite[Theorem~I.10.10]{goerss2009simplicial}. 

Let~$f,g\colon X\to Y$ be maps of simplicial sets.
A \emph{homotopy from~$f$ to~$g$} is a map $H \colon X \times 
\Delta^1 \to Y$ such that $H \circ i_0 = f$ and $H \circ i_1 = g$, where 
$i_0, i_1 \colon X \to X \times \Delta^1$ are the obvious inclusions. 
If $Y$ is a Kan complex, then the notion of homotopy of maps $X \to Y$ defines an equivalence relation~\cite[Corollary~I.6.2]
{goerss2009simplicial} and the same also holds for homotopies of maps relative to a subcomplex. The equivalence classes are called \emph{homotopy classes of maps} $X \to Y$. Given a Kan complex $X$, a basepoint~$x_0\in X_0$, and $n > 0$, the $n$-\emph{th simplicial homotopy group}~$\pi_n(X, x_0)$ of~$X$ is defined to be the set of homotopy classes of maps $(\Delta^n, \partial \Delta^n) \to (X, x_0)$. Moreover, we define $\pi_0(X)$ as the set of homotopy
 classes
of vertices (path components) in $X$. One can define a group operation on $\pi_n(X, x_0)$ for $n \geq 1$ and show that these groups are abelian for $n \geq 2$~\cite[Corollary~I.7.7]{goerss2009simplicial}.
The simplicial homotopy groups of a Kan complex~$X$ agree with the homotopy groups of its geometric realization~\cite[Proposition~3.6.3]{hovey2007model}:
\[
\pi_n(X, x_0) \cong \pi_n(|X|, |x_0|).
\]

Given a Kan complex~$X$, there are compatible maps $\tau_n \colon X \to P_n(X)$, $n \geq 0$, which fit into a diagram (tower)
\begin{equation}
\label{eqn:Postnikov}
X \to \cdots \to P_{n+1}(X) \to P_n(X) \to \cdots \to P_0(X)
\end{equation}
that models the Postnikov tower of $X$. We will consider a specific model due to Moore:

\begin{example}[Moore--Postnikov tower]\label{example:moore:postnikov}
There exists a model for the Postnikov tower of a Kan complex~$X$ with the following properties~\cite[Section~VI.3]{goerss2009simplicial}:
\begin{enumerate}[label=\enum]
   \item $P_n(X)$ is a Kan complex and $P_{n+1}(X) \to P_n (X)$ is a Kan fibration for every $n \geq 0$;
   \item the canonical map $X \to \mathrm{lim}_n P_n(X)$ is an isomorphism;
   \item the map $\tau_n \colon X \to P_n(X)$ is a surjective Kan fibration which is an isomorphism in simplicial degrees~$\leq n$;
    \item the map $\tau_n \colon X \to P_n(X)$ induces an isomorphism $$\pi_i(X, x) \xrightarrow{\cong} \pi_i(P_n(X), \tau_n(x))$$ for all $i \leq n$ and $x \in X$;
    \item we have $\pi_i(P_n(X), y) = 0$ for all $i > n$ and $y \in P_n(X)$.
\end{enumerate}
\end{example}

\subsection{Seminormed local coefficient systems}\label{subsec:seminormed local coeff}
    Throughout this article, we work over a normed ring $(R,|\cdot |)$, e.g.,~$\Z$ or~$\R$ with the usual norms. 
    A \emph{seminormed $R$-module~$(A, \|\cdot\|_A)$} is an $R$-module $A$ together with a map $\|\cdot \|_A\colon A\to \R_{\geq 0}$ satisfying
	\[
		\|r\cdot a_1+a_2 \|_A\le |r|\cdot \|a_1 \|_A+\|a_2\|_A
	\]
	for all~$r\in R$ and $a_1,a_2\in A$. 
 
Let~$\Mod^\sn_R$ denote the category whose objects are seminormed $R$-modules and whose morphisms are seminorm non-increasing $R$-linear maps, that is, $R$-linear maps $f \colon (A, \|\cdot\|_A) \to (B, \|\cdot\|_B)$ such that $\|f(a)\|_B \leq \|a\|_A$ for all $a \in A$. The forgetful functor to the category of $R$-modules (forgetting the norm) 
$$u\colon \Mod^\sn_R\to \Mod_R$$ 
admits a right adjoint $$\tr\colon \Mod_R\to \Mod^\sn_R$$ which equips an $R$-module $A$ with the trivial seminorm $\|\cdot\|_{A, \tr} \colon A \to \R_{\geq 0}$ that takes the value $0$ everywhere. The unit of the adjunction
\begin{equation}
\label{eqn:unit}
	\eta\colon \id_{\Mod^\sn_R}\Rightarrow \tr\circ u
\end{equation}
is given by the map $(A, \|\cdot\|_A)\to (A, \|\cdot \|_{A, \tr})$ that is the identity on~$A$. 

We recall that a \emph{groupoid} is a small category such that every morphism is an isomorphism. The category of groupoids $\Grpd$ is the full subcategory of the category of small categories $\Cat$ that is spanned by the groupoids. The inclusion functor $\Grpd \hookrightarrow \Cat$ admits a left adjoint $$(-)^{\gp} \colon \Cat \to \Grpd$$ which is defined on objects by inverting all the morphisms of a small category.

\begin{defi}[Fundamental groupoid of a simplicial set]
    The left adjoint of the composite adjunction
$$\pi \colon \SSet\overset{\tau}{\underset{N}\rightleftarrows} \Cat \overset{(-)^{\gp}}{\rightleftarrows} \Grpd \colon N$$
is the \emph{fundamental groupoid functor}.
\end{defi}

More explicitly, given a simplicial set $X$, the fundamental groupoid $\pi(X)$ has as objects the $0$-simplices $X_0$ of $X$ and is generated (as groupoid) by the $1$-simplices $X_1$ with a relation $\partial_0 \sigma \circ \partial_2 \sigma \sim \partial_1 \sigma$ for any $2$-simplex $\sigma \in X_2$. For a topological space~$E$, the fundamental groupoid of the singular set $\Sing(E)$ recovers the fundamental groupoid of $E$. For a Kan complex $X$, the group of automorphisms of a vertex $x_0 \in X_0$ in $\pi(X)$ is isomorphic to the \emph{fundamental group} $\pi_1(X, x_0)$ of $X$ at the basepoint $x_0$ -- and this agrees with the fundamental group $\pi_1(|X|, |x_0|)$ of the geometric realization~$|X|$ of $X$. 

\begin{defi}[Seminormed local coefficient system]
	Let~$X$ be a simplicial set.
	A \emph{seminormed local coefficient system~$\calA$ on~$X$} is a functor $\calA \colon \pi(X) \to \Mod^\sn_R$.
\end{defi}

Since every morphism in~$\pi(X)$ is invertible and the morphisms in $\Mod^\sn_R$ are norm non-increasing, it follows that $\calA$ sends every morphism to an isometric isomorphism. 
If $X$ is the singular set of a topological space, this definition agrees with the usual definition of a (seminormed) local coefficient system on a topological space~\cite{fauser2019parametrised}. 

The functoriality of the fundamental groupoid $\pi(-)$ yields a pullback operation for local coefficients: for a map~$f\colon X\to Y$ of simplicial sets and a seminormed local coefficient system~$\calB$ on~$Y$, we denote by $f^*\calB\coloneqq \calB\circ \pi(f)$ the seminormed local coefficient system on~$X$ that is defined by precomposition with the functor $\pi(f) \colon \pi(X) \to \pi(Y)$.

\begin{defi}[Category of simplicial sets with local coefficients]
\label{defi:category:Ldual}
	We define the category~$\mathcal{L}^*$ as follows: 
	\begin{itemize}
		\item the objects are pairs~$(X;\calA)$, where~$X$ is a simplicial set and~$\calA$ is a seminormed local coefficient system on~$X$;
		\item a morphism $(X;\calA) \to (Y; \calB)$ is a pair~$(f, \theta)$, where $f \colon X \to Y$
is a map of simplicial sets and $\theta \colon f^*\calB \Rightarrow \calA$ is a natural transformation of seminormed local coefficient systems on $X$;
\item the composition of $(f, \theta) \colon (X; \calA) \to (Y; \calB)$ and $(g, \eta) \colon (Y; \calB) \to (Z; \calC)$ is given by $g \circ f \colon X \to Z$ and 
$$(g \circ f)^* \calC = f^*(g^*\calC) \xRightarrow{f^*(\eta)} f^*\calB \xRightarrow{\theta} \calA,$$
where $f^*(\eta)$ denotes the pullback of the natural transformation $\eta$ along (precomposition with) $f$.
	\end{itemize}
\end{defi}

\begin{rem} 
Seminormed local coefficient systems are essentially the same as locally constant functors $\Delta\shortdownarrow X\to \Mod^\sn_R$. Indeed, there is a functor 
	\[
		e\colon \Delta \shortdownarrow X\to \pi(X)
	\]
that is given on objects by $e(\tau \colon \Delta^n \to X)=\tau(n)$. Here $n$ denotes that $0$-simplex of~$\Delta^n$ given by the map $[0] \to [n]$ that is defined by $n \in [n]$. On morphisms, the functor~$e$ sends $u \colon (\sigma \colon \Delta^n \to X) \to (\tau \colon \Delta^m \to X)$ to the morphism $\tau(u(n)) = \sigma(n) \to \tau(m)$ in~$\pi(X)$ defined by $\tau$. The functor~$e$ descends to an equivalence of groupoids~\cite[Theorem~III.1.1]{goerss2009simplicial} $$\overline{e}\colon (\Delta \shortdownarrow X)^\gp\xrightarrow{\sim} \pi(X).$$ Thus, we may view seminormed local coefficient systems (up to equivalence) as functors $\Delta \shortdownarrow X \to \Mod^{\sn}_R$ that send every morphism to an isomorphism.  
\end{rem}

\subsection{Bounded cohomology of simplicial sets with local coefficients}
\label{sec:bc}

	Let~$X$ be a simplicial set and let $\calA$ be a seminormed local coefficient system of $R$-modules on $X$.
    For every $n$-simplex $\sigma \in X_n$, we denote by $\sigma_{[0,k]} \in X_k$ the front $k$-face of~$\sigma$, that is, the restriction of $\sigma$ along the map $\iota_{k \leq n} \colon [k] \to [n], j\mapsto j$. 
    Of course $\sigma_{[0,k]}$ is simply the image of~$\sigma$ under the composition of face maps $\partial_{k+1}^{k+1} \circ \cdots \circ \partial_n^n \colon X_n \to X_k$.
 
	The \emph{bounded cochain complex~$C^\bullet_b(X;\calA)$ of~$X$ with local coefficients in~$\calA$} consists of the seminormed $R$-modules
	\[
		C_b^n(X; \calA) \coloneqq \bigl\{\phi \colon X_n\to \prod_{x\in X_0} \calA(x) \bigm\vert \phi(\sigma)\in \calA(\sigma_0), \ \ \|\phi\|_{\infty} < +\infty\bigr\},
	\]
	where $\|\phi\|_\infty\coloneqq \sup_{\sigma\in X_n}\|\phi(\sigma)\|_{\calA(\sigma_0)}$. 
	The boundary operator
$\delta^n \colon C^n_b(X; \calA) \to C^{n+1}_b(X; \calA)$ is given by the standard formula
\[
	\delta^n(\phi)(\sigma) \coloneqq \calA([\sigma_{[0, 1]}]^{-1})(\phi(\partial_0\sigma)) + \sum_{i = 1}^{n+1} (-1)^i \phi(\partial_i \sigma).
\]
It is easy to verify that $\|\delta^n(\phi)\|_\infty<+\infty$ if $\|\phi\|_\infty<+\infty$ and that $\delta^{n+1}\circ \delta^n=0$. Moreover, the bounded cochain complex is functorial with respect to morphisms in the category $\mathcal{L}^*$ taking values in seminormed $R$-cochain complexes, that is, the category of cochain complexes of seminormed $R$-modules with bounded boundary operators, and degreewise bounded cochain maps as morphisms. 

\begin{defi}[Bounded cohomology with local coefficients] Let $X$ be a simplicial set and let $\calA$ be a seminormed local coefficient system on $X$. The \emph{bounded cohomology of~$X$ with local coefficients in~$\calA$} is defined by
	\[
		H_b^*(X; \mathcal{A}) \coloneqq H^*(C_b^\bullet(X; \mathcal{A}), \delta^\bullet).
	\]
\end{defi}

These cohomology groups inherit a seminorm from the bounded cochain complex $C_b^{\bullet}(X; \mathcal{A})$. 
Bounded cohomology defines a functor into the category of graded seminormed $R$-modules
\[
	H^*_b\colon (\calL^*)^\op\to \mathrm{gr}\Mod^\sn_R.
\]
	
\begin{example}[Simplicial cohomology]
\label{ex:ordinary}
	Let~$\calB$ be a local coefficient system on~$X$ in the classical sense (without seminorms). Let $\calB_{\tr}$ denote the seminormed local coefficient system $\calB$ equipped with the trivial seminorm. Then every cochain is bounded and hence, $C^\bullet_b(X; \calB_{\tr})$ agrees with the usual simplicial cochain complex $C^\bullet(X; \calB)$ of the simplicial set $X$ with local coefficients in $\calB$. Hence bounded cohomology with \emph{trivially seminormed} coefficients recovers ordinary cohomology as a special case.
\end{example}	
	
	Let~$(X;\calA)$ be an object in~$\calL^*$. We denote by $\calA_{\un}$ the underlying local coefficient system of $R$-modules (after forgetting the seminorm) and by~$\calA_{\tr}$ the corresponding seminormed local coefficient system equipped with the trivial seminorm. There is an obvious morphism $(X;\calA_{\tr})\to (X;\calA)$ in~$\calL^*$ and the induced map of bounded cochain complexes 
	\[
		C^\bullet_b(X;\calA)\to C^\bullet_b(X;\calA_{\tr})=C^\bullet(X;\calA_{\un})
	\]
	is the inclusion of bounded cochains into all (not necessarily bounded) cochains. This is a natural transformation of functors on the category $\mathcal{L}^*$. Hence, this induces a natural transformation from bounded cohomology to ordinary cohomology, as functors on $\mathcal{L}^*$, 
	\[
		\comp^* \colon H^*_b(X;\calA)\to H^*(X;\calA_{\un})
	\]
 known as the \emph{comparison map}.

 \begin{rem}\label{rem:bc:spaces:and:bc:groups}
    The bounded cohomology of simplicial sets extends the bounded cohomology of topological spaces and of (discrete) groups~\cite{Frigerio:book, Ivanov17, ivanov2020simplicial}. 
    Given a topological space~$E$ and a seminormed local coefficient system~$\calA$ on~$E$, the bounded cohomology $H_b^*(\Sing(E); \calA)$ agrees by definition with the bounded cohomology $H_b^*(E; \calA)$ of the topological space~$E$ (see also Remark \ref{twisted-vs-local} below).
    Given a group~$\Gamma$ (viewed as a small category with one object) and a seminormed $R\Gamma$-module~$A$,
    let~$\calA$ be the seminormed local coefficient system on the nerve~$N(\Gamma)$ with $\calA(e)=A$ for the unique $0$-simplex~$e$ of~$N(\Gamma)$ and maps induced by the $\Gamma$-action on~$A$.
    Then the bounded cohomology $H_b^*(N(\Gamma);\calA)$ agrees with the bounded cohomology $H^*_b(\Gamma;A)$ of the group~$\Gamma$.
\end{rem}

\begin{rem}[Bounded cohomology of topological spaces with twisted coefficients] \label{twisted-vs-local}
For topological spaces, the relationship between bounded cohomology with local coefficients and bounded cohomology with twisted coefficients is as in ordinary cohomology~\cite[Theorem~VI.3.4]{whiteheadbook}. 
Let~$E$ be a path-connected topological space with basepoint~$e_0$ that admits a universal covering~$\widetilde{E}$.
	The fundamental group~$\pi_1(E,e_0)$ acts on the $R$-modules~$R[\Sing_\bullet\widetilde{E}]$ of singular chains. 
    Let~$A$ be a seminormed $R[\pi_1(E,e_0)]$-module.
    
	The \emph{bounded cohomology~$H^{\bullet}_b(E;A)$ of~$X$ with twisted coefficients in~$A$} is defined to be the cohomology of the cochain complex
	\[
		C^\bullet_b(E;A)\coloneqq \bigl\{\varphi\in \Hom_{\pi_1(E,e_0)}(R[\Sing_\bullet \widetilde{E}], A) \bigm\vert \|\varphi\|_\infty<+\infty \bigr\}
	\]
equipped with the usual differential.
The category of seminormed $R[\pi_1(E,e_0)]$-modules is equivalent to the category of seminormed local coefficient systems on~$E$.
Indeed, given a seminormed local coefficient system~$\calA$ on~$E$, the fundamental group~$\pi_1(E,e_0)$ acts on the seminormed $R$-module~$\calA(e_0)$ and this canonically determines $\calA$ up to canonical isomorphism~\cite[Example~1.3.2]{fauser2019parametrised}. The universal covering map $\widetilde{E}\to E$ induces an isometric isomorphism of cochain complexes
	\[
		 C^\bullet_b(\Sing(E);\calA)\xrightarrow{\cong} C^\bullet_b(E;\calA(e_0))
	\]
	and hence an isometric isomorphism of cohomology groups $H^*_b(\Sing(E);\calA)\cong H^*_b(E;\calA(e_0))$. 
\end{rem}

Similarly to ordinary cohomology, the zeroth bounded cohomology group is given by $\pi_1$-invariants:

\begin{prop}[Bounded cohomology in degree zero]
\label{prop:degree zero}
	Let~$X$ be a connected simplicial set with basepoint~$x_0$ and let~$\calA$ be a seminormed local coefficient system on~$X$.
	Then there is an isomorphism of seminormed $R$-modules
	\[
		H^0_b(X;\calA)\cong \calA(x_0)^{\pi_1(X,x_0)},
	\]
	where~$\calA(x_0)^{\pi_1(X,x_0)}$ denotes the $\pi_1(X, x_0)$-fixed points of the~$R[\pi_1(X,x_0)]$-module $\calA(x_0)$.
\end{prop}
\begin{proof}(sketch) A $0$-cochain is a bounded collection $(\phi_x)_{x \in X_0} \in \prod_{x \in X_0} \calA(x)$. This is a $0$-cocycle precisely when for every $1$-simplex $\sigma \in X_1$, we have $\calA([\sigma]^{-1})(\phi_{\partial_0 \sigma}) = \phi_{\partial_1 \sigma}$. Let $R \colon \pi(X) \to \Mod^{\sn}_R$ be the constant functor at the $R$-module $R$. Then, a $0$-cocycle is given exactly by a natural transformation (``point'') $R \Rightarrow \calA$. By connectivity (of $\pi(X)$), this is determined by its restriction to $\pi_1(X, x_0)$, thus yielding a $\pi_1(X, x_0)$-fixed point of $\calA(x_0)$. 
\end{proof}

As a consequence, the zeroth bounded cohomology group~$H_b^0$ is invariant under weak equivalences of simplicial sets (not necessarily Kan complexes) -- we postpone the precise definitions and a detailed discussion of homotopy invariance to Section~\ref{sec:weak invariance bc}.

\subsection{Bounded cohomology and bounded products}\label{subsec:bounded products}
The interaction between bounded cohomology and (bounded) products will play an important role in later sections. 
The situation is different from ordinary cohomology (which simply sends arbitrary coproducts of simplicial sets to products) and, in some sense, the comparison between bounded cohomology and ordinary cohomology can be reduced to their different properties with respect to infinite coproducts. 

\begin{defi}[Bounded product of seminormed $R$-modules]
Let~$(C_i)_{i\in I}$ be a collection of seminormed $R$-modules.
The \emph{bounded product of~$(C_i)_{i\in I}$} is the $R$-submodule~$\bprod{i\in I} C_i$ of~$\prod_{i\in I} C_i$ defined by
	\[
		\bprod{i\in I} C_i\coloneqq \bigl\{ (c_i)_{i} \in \prod_{i\in I} C_i \bigm\vert \sup_{i\in I} \|c_i\|_{C_i}<+\infty \bigr\}
	\]
equipped with supremum seminorm. 
\end{defi}

\begin{defi}[Bounded product of a uniform collection]
Let~$((C^\bullet_i,\delta^\bullet_i))_{i\in I}$ be a collection of seminormed $R$-cochain complexes.
We say that~$((C^\bullet_i,\delta^\bullet_i))_{i\in I}$ is \emph{uniform} if for every~$n\in \Z$ the supremum~$\sup_{i\in I}\|\delta^n_i\|$ is finite.

The \emph{bounded product} of a uniform collection~$((C^\bullet_i,\delta^\bullet_i))_{i\in I}$ of seminormed $R$-cochain complexes is the seminormed $R$-cochain complex $\bprod{i\in I}(C^\bullet_i, \delta^{\bullet}_i)$ that is given degreewise by the bounded products of cochain modules and the products of differentials.
\end{defi}

\begin{example}
Every collection of bounded cochain complexes of simplicial sets is uniform.
\end{example}

Given a uniform collection~$((C^{\bullet}_i, \delta^{\bullet}_i))_{i \in I}$ of seminormed $R$-cochain complexes, there is a canonical map of seminormed $R$-modules for every $n \in \Z$:
\begin{equation} \label{comparison-map-Phi}
\Phi^n \colon H^n\Big(\bprod{i\in I} C^\bullet_i\Big) \to \bprod{i\in I}H^n(C^\bullet_i).
\end{equation}
The map $\Phi^n$ is surjective and of seminorm $\leq 1$, but it is not injective in general. Indeed, the kernel clearly contains cohomology classes represented by elements of the form $(\delta^{n-1}_i(c_i))_{i \in I}$ but these are not in general coboundaries in the bounded product of the cochain complexes. 
Thus, unlike ordinary cohomology, bounded cohomology does not preserve (bounded) products in general.

\begin{defi}[Bounded boundary condition for a uniform collection]
\label{defn:Phi iso}
Let $((C^{\bullet}_i, \delta^{\bullet}_i))_{i \in I}$ be a uniform collection of seminormed $R$-cochain complexes and let $n \in \Z$. We say that 
$((C^\bullet_i,\delta^\bullet_i))_{i\in I}$ satisfies the \emph{bounded boundary condition in degree~$n$} ($\BBC^n$) if the map $\Phi^n$ \eqref{comparison-map-Phi} is bijective. 
\end{defi}

\begin{rem}
If the map $\Phi^n$ \eqref{comparison-map-Phi} is bijective, then two (bounded) $n$-cocycles $(\phi_i)_{i \in I}$ and $(\psi_i)_{i \in I}$ in the bounded product represent the same $n$-cohomology class if and only if the $n$-cocycles $\phi_i$ and $\psi_i$ represent the same $n$-cohomology class for each $i \in I$. Moreover, then~$\Phi^n$ is actually an isometric isomorphism between seminormed $R$-modules. 
\end{rem}

\begin{example}
\label{ex:finite collections}
Every \emph{finite} collection of seminormed cochain complexes satisfies the bounded boundary condition in all degrees.
\end{example}

The bounded boundary condition is closely related to the \emph{uniform uniform boundary condition}~\cite[Appendix~A]{liloehmoraschini}, which is a many-object version of the usual uniform boundary condition from the work of Matsumoto--Morita~\cite{Matsu-Mor}.

\begin{defi}[Uniform boundary condition for a cochain complex, \cite{Matsu-Mor}]
\label{defn:UBC}
Let~$(C^\bullet,\delta^\bullet)$ be a seminormed $R$-cochain complex, let~$n\in \Z$, and let~$\kappa\in \R_{\ge 0}$.
We say that $(C^\bullet,\delta^\bullet)$ satisfies the \emph{uniform boundary condition in degree~$n$ with constant~$\kappa$} ($\kappa$-$\UBC^n$) if for every
$b\in \im(\delta^{n-1}\colon C^{n-1} \to C^n)$
there exists~$c\in C^{n-1}$ with~$\delta^{n-1}(c)=b$ and
	\[
		\|c\|\le \kappa\cdot \|b\|.
	\]
We say that~$(C^\bullet,\delta^\bullet)$ satisfies the \emph{uniform boundary condition in degree~$n$} ($\UBC^n$) if it satisfies $\kappa$-$\UBC^n$ for some~$\kappa\in \R_{\ge 0}$.
\end{defi}

\begin{prop}
\label{thm:matsumoto morita}
Let $(C^\bullet,\delta^\bullet)$ be a seminormed $R$-cochain complex and let~$n \in \Z$.
Then $(C^\bullet,\delta^\bullet)$ satisfies~$\UBC^n$ if and only if the induced bijective bounded linear map of seminormed $R$-modules 
$\overline{\delta}^{n-1} \colon C^{n-1}/\ker(\delta^{n-1}) \rightarrow \im(\delta^{n-1})$
has a bounded $R$-linear inverse. 
\end{prop}

By the open mapping theorem, the uniform boundary condition for cochain complexes of Banach $\R$-modules can be characterized in terms of cohomology. 

\begin{cor}[cf.~{\cite[Theorem~2.3]{Matsu-Mor}}] \label{cor:matsumoto-morita}
Let~$(C^\bullet,\delta^\bullet)$ be a seminormed $\R$-cochain complex consisting of Banach modules and let~$n \in \Z$. Then $(C^\bullet,\delta^\bullet)$ satisfies~$\UBC^n$ if and only if the seminormed $\R$-module~$H^n(C^\bullet,\delta^\bullet)$ is Banach. 

In particular, if~$H^n(C^\bullet,\delta^\bullet)=0$, then~$(C^\bullet,\delta^\bullet)$ satisfies~$\UBC^n$.
\end{cor}

The following is a useful criterion for the bounded boundary condition of a uniform collection of seminormed cochain complexes 
in terms of the uniform boundary condition of its members.

\begin{prop}[{\cite[Theorem~A.15]{liloehmoraschini}}]
\label{prop:bounded products}
Let~$((C^\bullet_i,\delta^\bullet_i))_{i\in I}$ be a uniform collection of seminormed $R$-cochain complexes, let $n \in \Z$, and let $\kappa\in \R_{\ge 0}$.
If~$(C^\bullet_i,\delta^\bullet_i)$ satisfies $\kappa$-$\UBC^n$ for every $i\in I$, then the collection $((C^{\bullet}_i, \delta^{\bullet}_i))_{i \in I}$ satisfies~$\BBC^n$. 
\end{prop}

For a collection of normed $\R$-cochain complexes, we prove the converse of Proposition~\ref{prop:bounded products} for all but finitely many members. 

\begin{prop}
\label{prop:BBC iff UUBC}
	Let~$((C^\bullet_i,\delta^\bullet_i))_{i\in I}$ be a uniform collection of normed $\R$-cochain complexes and let $n \in \Z$.
	The following are equivalent:
	\begin{enumerate}[label=\enum]
		\item the collection $((C^{\bullet}_i, \delta^{\bullet}_i))_{i \in I}$ satisfies~$\BBC^n$;
		\item there exists a finite subset~$J\subset I$ and $\kappa\in \R_{\ge 0}$ such that $(C^\bullet_i,\delta^\bullet_i)$ satisfies $\kappa$-$\UBC^n$ for every $i\in I\setminus J$.
	\end{enumerate} 
	\begin{proof}
		First note that $\Phi^n$ is injective if and only if for every bounded element of the form $(\delta^{n-1}_i(c_i))_{i \in I} \in \bprod{i\in I}C_i^n$, there is a bounded element $(x_i)_{i \in I} \in \bprod{i \in I} C_i^{n-1}$ such that $\delta^{n-1}_i(x_i) = \delta^{n-1}_i(c_i)$ for all~$i\in I$.
		Then it is clear that~(ii) implies~(i).
		
		Conversely, suppose that (ii) does not hold.
		We construct inductively an infinite sequence of elements $b_{i_k}\in C_{i_k}^{n}$, $k\ge 1$, with~$i_k\neq i_l$ for $k \neq l$. 
		Suppose that~$b_{i_1},\ldots,b_{i_{k-1}}$ have already been constructed, $k \geq 1$.
		For~$J=\{i_1,\ldots,i_{k-1}\}$ and~$\kappa=k$, there exists~$i_k\in I\setminus \{i_1,\ldots,i_{k-1}\}$ and an element~$b_{i_k}\in \im(\delta_{i_k}^{n-1}\colon C_{i_k}^{n-1}\to C_{i_k}^n)$ with $\|b_{i_k}\|=1$ such that for every~$c\in C_{i_k}^{n-1}$ with $\delta_{i_k}^{n-1}(c)=b_{i_k}$, we have $\|c\|>k$. 
        Here we use that~$C_{i_k}^n$ is a normed $\R$-module.
        Expanding this sequence of elements to an element of $\bprod{i\in I}C_i^n$ by inserting the zero elements where needed, we obtain a non-trivial cohomology class of $\bprod{i\in I}C_i^{\bullet}$ which is in the kernel of $\Phi^n$. 
	\end{proof}
\end{prop}

\begin{example}
Let~$(C^\bullet,\delta^\bullet)$ be a normed $\R$-cochain complex and consider an infinite collection $((C^{\bullet}_i, \delta^{\bullet}_i))_{i\in I}$ of normed $\R$-cochain complexes that consists of copies of $(C^\bullet,\delta^\bullet)$. Then, by Proposition~\ref{prop:BBC iff UUBC}, 
$((C^{\bullet}_i, \delta^{\bullet}_i))_{i\in I}$ satifies $\BBC^n$ if and only if $(C^\bullet,\delta^\bullet)$ satisfies $\UBC^n$.

Let $X$ be a simplicial set whose bounded $\R$-cochain complex (with constant coefficients in $\R$) does not satisfy $\UBC^n$ for some $n \in \Z$. For example, the bounded $\R$-cochain complex of the singular set of the wedge of two circles does not satisfy $\UBC^3$ \cite{Soma}. Then an infinite collection consisting of copies of the bounded cochain complex of $X$ does not satisfy $\BBC^n$. 
\end{example}

\section{Homotopy invariance of bounded cohomology}
\label{sec:weak invariance bc}

Ivanov~\cite{Ivanov17} proved that the bounded cohomology of topological spaces (with constant coefficients in $\R$) is invariant under weak homotopy equivalences (see also \cite{ivanov2020simplicial}). In this section, we investigate the homotopy invariance properties of bounded cohomology with seminormed local coefficients in the more general context of simplicial sets.

\subsection{Homotopy equivalences in \texorpdfstring{$\calL^*$}{L*}}

Let $(X; \calA)$ be an object in $\calL^*$. The simplicial cylinder object 
is given by
$$X \sqcup X = X \times \partial \Delta^1 \stackrel{i}{\hookrightarrow} X \times \Delta^1 \xrightarrow{p} X.$$
The simplicial set $X \times \Delta^1$ is equipped with the seminormed local coefficient system $\calA \times \Delta^1 \coloneqq p^*(\calA)$. We obtain an associated cylinder object in $\calL^*$
$$(X \sqcup X; \calA \sqcup \calA) \stackrel{i}{\hookrightarrow} (X \times \Delta^1; \calA \times \Delta^1) \xrightarrow{p} (X; \calA).$$
Let~$f=(f,\theta),f'=(f',\theta')\colon (X;\calA)\to (Y;\calB)$ be maps in~$\calL^*$.
A \emph{homotopy from~$f$ to~$f'$}
is a map in~$\calL^*$
$$H = (H, \Theta) \colon (X \times \Delta^1; \calA \times \Delta^1) \to (Y; \calB)$$ 
such that $f = H \circ i_0$ and $f' = H \circ i_1$, where the inclusions $i_0, i_1$ are defined to be the identity on the local coefficients. At the level of local coefficients, the simplicial map~$H$ induces a natural isomorphism $h \colon f^*\calB \Rightarrow f'^* \calB$ which corresponds precisely to $H^*\calB$. Then the natural transformation $\Theta \colon H^*\calB \Rightarrow \calA \times \Delta^1$ corresponds precisely to the compatibility condition $\theta = \theta' \circ h$.

As in the underlying theory of simplicial sets, it is not possible in general to compose two homotopies in~$\calL^*$.
We say that $f$ and $f'$ are \emph{homotopic in $\calL^*$}, written $f \simeq_{\calL^*} f'$, if there is a zigzag of homotopies in~$\calL^*$
connecting $f$ and $f'$. A map $f=(f, \theta) \colon (X; \calA) \to (Y; \calB)$ is a \emph{homotopy equivalence} 
if there exists a map $g = (g, \eta) \colon (Y; \calB) \to (X; \calA)$ such that $f\circ g \simeq_{\calL^*} \id_{(Y; \calB)}$ and $g \circ f \simeq_{\calL^*} \id_{(X; \calA)}$.

The situation simplifies when the target $Y$ is a Kan complex:

\begin{prop} \label{homotopy-simplification}
Let $Y$ be a Kan complex and suppose that the maps $$(f, \theta), (f', \theta') \colon (X; \calA) \to (Y; \calB)$$ are homotopic in $\calL^*$. Then there is a homotopy in $\calL^*$ 
from $(f, \theta)$ to $(f', \theta')$.
\begin{proof}
    There is a zigzag of homotopies connecting $f$ to $f'$ and a compatible natural isomorphism $f^*\calB \Rightarrow f'^*\calB$ (by inverting the isomorphisms in the zigzag of natural isomorphisms). Then it is possible to find a homotopy from $f$ to $f'$~\cite[Corollary~I.6.2]{goerss2009simplicial} that induces the same natural isomorphism on fundamental groupoids as the zigzag of homotopies does.
\end{proof}
\end{prop}

In addition, we have the following characterization of homotopy equivalences in~$\calL^*$.

\begin{prop}\label{prop:characterisation:simpl:homotopy:eq}
If $(f, \theta) \colon (X; \calA) \to (Y; \calB)$ is a homotopy equivalence in $\calL^*$, then~$f$ is a homotopy equivalence of simplicial sets and $\theta$ is a natural isomorphism. 

Conversely, let $(f, \theta) \colon (X; \calA) \to (Y; \calB)$ be a map in $\calL^*$, where $X$ and $Y$ are Kan complexes. If $f \colon X \to Y$ is a homotopy equivalence of simplicial sets and $\theta \colon f^*\calB \Rightarrow \calA$ is a natural isomorphism, then $(f, \theta) \colon (X; \calA) \to (Y; \calB)$ is a homotopy
equivalence in $\calL^*$. 
\end{prop}
\begin{proof}
Suppose that $(f, \theta)$ is a homotopy equivalence in $\calL^*$. Clearly, $f$ is a homotopy equivalence and $\pi(f) \colon \pi(X) \to \pi(Y)$ is an equivalence of groupoids. It remains to prove that $\theta \colon f^* \calB \Rightarrow \calA$ is a natural isomorphism. Let $(g, \eta) \colon (Y; \calB) \to (X; \calA)$ in $\calL^*$ be a homotopy inverse, so there are (zigzags of) homotopies $g \circ f \simeq_{\calL^*} \id_{(X; \calA)}$ and $f \circ g \simeq_{\calL^*} \id_{(Y; \calB)}$. At the level of local coefficients, these yield (zigzags of) natural isomorphisms 
$(g \circ f)^* \calA \cong \calA$ and $(f \circ g)^*\calB \cong \calB$. Moreover, these natural isomorphisms that arise from the homotopies agree with $\theta \circ f^*(\eta)$ and $\eta \circ g^*(\theta)$, respectively. Then note that $\eta \colon g^*\calA \Rightarrow \calB$ (resp.\ $\theta \colon f^*\calB \Rightarrow \calA$) is a natural isomorphism if and only if $f^*(\eta) \colon f^*g^*\calA \Rightarrow f^*\calB$ (resp.\ $g^*(\theta)$) is a natural isomorphism and the result follows.

Conversely, suppose that $f \colon X \to Y$ is a homotopy equivalence of Kan complexes. Then there exist a homotopy inverse $g \colon Y \to X$ and homotopies $H \colon \id_X  \simeq g \circ f$ and $H' \colon f \circ g \simeq \id_Y$~\cite[Corollary~I.6.2]{goerss2009simplicial}. 
Moreover, since $X$ and $Y$ are Kan complexes (= $\infty$-groupoids), we may assume that these homotopies are \emph{adjoint} in the following sense: the respective composite homotopies 
$$f \stackrel{f H}{\simeq} f \circ g \circ f \stackrel{H'f}{\simeq} f$$
and
$$g \stackrel{Hg}{\simeq} g \circ f \circ g \stackrel{gH'}{\simeq} g$$
are homotopic respectively to the trivial homotopy (\emph{triangle identities}). This is done by choosing $H$ and $H'$ to be the unit and counit transformations of the adjoint equivalence $f \colon X \rightleftarrows Y \colon g$ between $\infty$-groupoids. As a consequence, $\pi(f) \colon \pi(X) \leftrightarrows \pi(Y) \colon \pi(g)$ is also an adjoint equivalence (with respect to the unit and counit transformations that are induced by~$H$ and~$H'$). These specifications are necessary in order to extend the homotopy equivalence~$f$ to a homotopy equivalence in $\calL^*$. 
We may extend $g \colon Y \to X$ to a map $(g, \eta) \colon (Y; \calB) \to (X; \calA)$ in $\calL^*$ by setting $\eta \colon g^*\calA \Rightarrow \calB$ to be the natural isomorphism such that
$$g^*f^* \calB \xRightarrow{g^*(\theta)} g^*\calA \xRightarrow{\eta} \calB$$ is the one 
induced by $\pi(f) \circ \pi(g) \cong \id_{\pi(Y)}$. In other words, considering the induced adjoint pair $(g^*, f^*)$, the natural transformation~$\eta$ is the adjoint of the inverse of $\theta$. Then the composite
$$f^*g^*\calA \xRightarrow{f^*(\eta)} f^*\calB \xRightarrow{\theta} \calA$$ is also induced by $\pi(g) \circ \pi(f) \cong \id_{\pi(X)}$. This uses the triangle identities for the adjoint pair $(g^*, f^*)$ in order to identify $f^*(\eta)$ with the inverse of the composite 
$$f^*\calB \Rightarrow f^*g^*f^*\calB \xRightarrow{f^*g^*(\theta)} f^*g^*\calA$$
where the first natural transformation is induced by the counit transformation of $(g^*, f^*)$, and the latter agrees by naturality with 
$$f^* \calB \stackrel{\theta}{\Rightarrow} \calA \Rightarrow f^*g^*\calA.$$
Finally, $H$ (and similarly $H'$) extends to a homotopy in $\calL^*$ by setting $$\Theta \colon H^* \calA \Rightarrow \calA \times \Delta^1$$ to correspond with the natural isomorphism $(g \circ f)^*\calA \Rightarrow \calA$ given by $\pi(g) \circ \pi(f) \cong \id_{\pi(X)}$. This completes the proof that $(f, \theta) \colon (X; \calA) \to (Y; \calB)$ is a homotopy equivalence in $\calL^*$. 
\end{proof}

The proof in fact shows that the converse statement of Proposition \ref{prop:characterisation:simpl:homotopy:eq} holds more generally for homotopy equivalences of simplicial sets such that the corresponding homotopies to the identity maps can be chosen to be adjoint. It is well known that a map $f \colon X \to Y$ of Kan complexes is a \emph{weak equivalence} (that is, $|f|$ is a weak homotopy equivalence of topological spaces) if and only if $f$ is a homotopy equivalence \cite[Theorem~I.1.10]{goerss2009simplicial}. We conclude: 

\begin{cor}\label{cor:simpl:homotop:equi:iff:weakequiv:and:iso}
Let $X$ and $Y$ be Kan complexes. A map $(f, \theta) \colon (X; \calA) \to (Y; \calB)$ is a homotopy
 equivalence in $\calL^*$ if and only if $f$ is a weak equivalence and $\theta$ is a natural isomorphism.
\end{cor}

\subsection{Homotopy invariance of bounded cohomology} The bounded cohomology of topological spaces is invariant under weak homotopy equivalences \cite{Ivanov17}. In this subsection, we show a refinement of 
this property for bounded cohomology as a functor on the category $\calL^*$.

\begin{lemma}\label{lemma:simpl:hom:maps:induce:bounded:chain:}
    Let~$(f,\theta),(f',\theta')\colon (X;\calA)\to (Y;\calB)$ be maps in~$\calL^*$.
    If there is a homotopy~$(G,\Theta)$ from~$(f,\theta)$ to~$(f',\theta')$, then the cochain maps $C^\bullet_b(f, \theta)$ and $C^\bullet_b(f', \theta')$ are cochain homotopic via a cochain homotopy~$H$ whose norm in degree~$n$ is bounded by $n$ for every~$n\ge 0$.
\end{lemma}
\begin{proof}
In order to define the cochain homotopy $H$ we consider the usual triangulation of the simplicial prism $\Delta^{n-1} \times \Delta^1$.
We have $n$ non-degenerate $n$-simplices $p_j \colon \Delta^{n} \to \Delta^{n-1} \times \Delta^1$ -- these are the maps of posets $[n] \to [n-1] \times [1]$ from $(0,0)$ to $(n-1, 1)$ that move exactly $n-1$ steps to the right and once up. Hence, for every $j\in \{0, \ldots, n-1\}$ and every $(n-1)$-simplex $\sigma \in X_{n-1}$, we consider the following composition:
   \[
   \Delta^{n} \xrightarrow{p_j} \Delta^{n-1} \times \Delta^1 \xrightarrow{\sigma \times \textup{id}} X \times \Delta^1 \xrightarrow{G} Y
   \]
which defines an $n$-simplex in $Y$. Then, we define our desired cochain homotopy from $C^{\bullet}_b(f, \theta)$ to $C^{\bullet}_b(f', \theta')$ by
   \begin{align*}
   H^n \colon C^n_b(Y; \calB) &\to C^{n-1}_b(X; \calA)\\
   \varphi &\mapsto (\sigma\mapsto \sum_{j = 0}^{n-1} (-1)^j \varphi(G \circ (\sigma \times \textup{id}) \circ p_j)).
   \end{align*}
By definition of~$G$, the cochain homotopy $H$ respects the chosen local coefficients. Moreover, the norm of~$H^n$ is bounded by $n$.
\end{proof}

It follows that bounded cohomology is homotopy invariant:

\begin{cor}\label{cor:homotopic:map:same:map:in:bc}
    Two homotopic maps in~$\calL^*$ induce the same morphism in bounded cohomology.
    In particular, a homotopy equivalence in $\calL^*$ induces an isometric isomorphism in bounded cohomology.
\end{cor}

Corollary~\ref{cor:homotopic:map:same:map:in:bc} and Corollary~\ref{cor:simpl:homotop:equi:iff:weakequiv:and:iso} together yield:
\begin{cor}\label{cor:weak invariance bc}
Let $(f, \theta) \colon (X; \calA) \to (Y; \calB)$ be a map in $\calL^*$ such that $f$ is a weak equivalence between Kan complexes and $\theta$ is a natural isomorphism. Then $(f,\theta)$ induces an isometric isomorphism in bounded cohomology. 
\end{cor}

\begin{example} \label{example:G-vs-BG}
Let $\Gamma$ be a group (viewed as a groupoid) and let $B\Gamma = |N(\Gamma)|$ be the standard model for the classifying space of $\Gamma$ -- this is an Eilenberg--MacLane space of type $K(\Gamma,1)$~\cite[Example~I.1.5 and Proposition~I.7.8]{goerss2009simplicial}. Since $N(\Gamma)$ is a Kan complex~\cite[Lemma I.3.5]{goerss2009simplicial}, there is a weak equivalence $f \colon N(\Gamma) \to \Sing(B\Gamma)$ of Kan complexes, which is given by the adjoint of the identity map for~$B\Gamma$. The equivalence $\pi(f)$ is simply given by the full inclusion of the fundamental group at the canonical basepoint. Hence, by Corollary~\ref{cor:weak invariance bc}, $f$ induces an isometric isomorphism between $H^{\bullet}_b(N(\Gamma); f^*\calA)$ (bounded cohomology of the group $\Gamma$) and $H^{\bullet}_b(B\Gamma; \calA)$ (bounded cohomology of the topological space $B\Gamma$) for any seminormed local coefficient system $\calA$ (see Remark~\ref{rem:bc:spaces:and:bc:groups}).
\end{example}

\begin{rem}
Corollary~\ref{cor:weak invariance bc} yields the corresponding weak homotopy invariance of the bounded cohomology of topological spaces (cf.~\cite[Section 2]{Raptis-bounded+htpy}).
Indeed, given a weak homotopy equivalence $f \colon E \to E'$ of topological spaces, the singular set functor sends~$f$ to a weak equivalence of Kan complexes $\Sing(f) \colon \Sing(E) \to \Sing(E')$, since the canonical map $|\Sing(E)| \to E$ is a weak homotopy equivalence for any topological space~$E$~\cite{milnor1957geometric}.
Then the weak homotopy invariance of bounded cohomology follows from Corollary~\ref{cor:weak invariance bc} and Remark~\ref{rem:bc:spaces:and:bc:groups}.
\end{rem}

Further, we will need to control $\UBC$-constants along homotopy equivalences of Kan complexes:

\begin{cor} \label{cor:control_ubc_constant}
    Let $(f, \theta) \colon (X; \calA) \to (Y; \calB)$ be a homotopy equivalence in~$\calL^*$, where~$Y$ is a Kan complex.
    Let~$n\in \N$ and let~$\kappa\in \R_{\ge 0}$.
    If~$C_b^\bullet(X;\calA)$ satisfies $\kappa$-$\UBC^n$, then~$C_b^\bullet(Y;\calB)$ satisfies $(\kappa+n)$-$\UBC^n$.
    \begin{proof}
        Let~$(g,\eta)\colon (Y;\calB)\to (X;\calA)$ be a homotopy inverse of~$(f,\theta)$.
        Since~$Y$ is a Kan complex, there exists a homotopy from~$(f,\theta)\circ (g,\eta)$ to~$\id_{(Y;\calB)}$ (Proposition~\ref{homotopy-simplification}).
        Then the claim follows from Lemma~\ref{lemma:simpl:hom:maps:induce:bounded:chain:} and standard computations for $\UBC$-constants~\cite[Proposition~A.3]{liloehmoraschini}.
    \end{proof}
\end{cor}

The following is an immediate consequence of Corollary~\ref{cor:control_ubc_constant} and Proposition~\ref{prop:bounded products}.

\begin{cor}
\label{cor:weak invariance UBC}
Let~$((X_i;\calA_i))_{i\in I}$ be a collection of objects in~$\calL^*$, whose underlying simplicial sets are Kan complexes, that are pairwise  homotopy equivalent in $\calL^*$.
If for some~$i_0\in I$, the seminormed $R$-cochain complex~$C^\bullet_b(X_{i_0};\calA_{i_0})$ satisfies~$\UBC^n$, then the uniform collection of seminormed $R$-cochain complexes $(C^\bullet_b(X_i;\calA_i))_{i\in I}$ satisfies~$\BBC^n$.
\end{cor}

We recall that a (non-empty) Kan complex is $n$-\emph{connected}, $n\ge 0$, 
if it has trivial homotopy groups in all degrees $\leq n$. Then, we have:

\begin{cor}
\label{cor:n-connected bc}
Let $n \geq 1$, let~$X$ be an $n$-connected Kan complex, and let~$\calA$ be a seminormed local coefficient system on~$X$. Then the following hold:
\begin{enumerate}[label=\enum]
\item $H^i_b(X;\calA)=0$ for all~$i\in \{1,\ldots,n\}$;
\item $C^\bullet_b(X;\calA)$ satisfies~$\UBC^i$ for all~$i\le n$.
\end{enumerate}
\end{cor}
\begin{proof}
(i) By construction of the Moore model for the Postnikov tower of~$X$, the map $\tau_n \colon X \to P_n(X)$ is an isomorphism in simplicial degrees $\leq n$ and surjective in degree $n+1$ (Example~\ref{example:moore:postnikov}). 
Hence $\tau_n$ induces an isomorphism on bounded cohomology in degrees $\leq n$. On the other hand, $P_{n}(X)$ is weakly equivalent to $\Delta^0$, so it has trivial bounded cohomology by Corollary~\ref{cor:weak invariance bc}. 

(ii) Since $P_n (X)$ is weakly equivalent to $\Delta^0$, it satisfies $\UBC^i$ for all $i \geq 0$ by Corollary \ref{cor:control_ubc_constant}. The result follows because the bounded cochain complex of $P_n(X)$ agrees with the bounded cochain complex of $X$ in degrees $\leq n$.
\end{proof}

\subsection{Local coefficient system associated to a Kan fibration}
\label{sec:coefficient fibres}

A key example of a seminormed local coefficient system comes from the bounded cohomology of the fibers of a Kan fibration. The construction is analogous to the case of ordinary cohomology. 

Let~$f\colon X\to Y$ be a map of simplicial sets. For every $p \geq 0$, we consider the induced map of simplicial sets
\[
	f^{\Delta^p}\colon X^{\Delta^p}\to Y^{\Delta^p},
\]
where $X^{\Delta^p}$ and $Y^{\Delta^p}$ denote the (simplicial) mapping spaces. Given a $p$-simplex $\tau\colon \Delta^p\to Y$ of~$Y$, we denote the corresponding fiber of $f^{\Delta^p}$ by $F_\tau\coloneqq (f^{\Delta^p})^{-1}(\tau)$. 
This yields a functor
\[
	F_\bullet\colon (\Delta \shortdownarrow Y)^{\op} \to \SSet, \quad \tau\mapsto F_\tau.
\]
Let~$\calA$ be a seminormed local coefficient system on~$X$. Consider the composition
\[
    \iota_\tau\colon F_\tau\hookrightarrow X^{\Delta^p}\to X
\]
where the second map is given by restriction to/evaluation at the last $0$-simplex of~$\Delta^p$, given by $[0] \to [p]$, $0 \mapsto p$. Denote by~$\calA_\tau\coloneqq \iota_\tau^*\calA$ the restricted local coefficient system on~$F_\tau$.
This yields a refinement of the functor~$F_\bullet$ with values in~$\calL^*$
\[
	(F_\bullet;\calA_\bullet)\colon (\Delta\shortdownarrow Y)^{\op} \to \calL^*, \ \ \tau\mapsto (F_\tau;\calA_\tau).
\]
In detail, a morphism~$\alpha\colon \tau\to \tau'$ in~$\Delta\shortdownarrow Y$ induces a triangle
\[\begin{tikzcd}
	F_{\tau'} \ar{rr}{F(\alpha)}\ar{dr}[swap]{\iota_{\tau'}} && F_{\tau} \ar{dl}{\iota_{\tau}} \\
	& X
\end{tikzcd}\]
which commutes up to a canonical homotopy $\iota_{\tau} \circ F(\alpha) \simeq \iota_{\tau'}$ between evaluation functors at different $0$-simplices. The canonical homotopy gives rise to a natural isomorphism of coefficient systems on~$F_{\tau'}$ 
\[
F(\alpha)^* \calA_\tau = (\iota_{\tau}\circ F(\alpha))^*\calA\Rightarrow \iota_{\tau'}^*\calA = \calA_{\tau'}.\] 
Thus we obtain a map $(F_{\tau'};\calA_{\tau'})\to (F_{\tau};\calA_{\tau})$ in~$\calL^*$. Passing to bounded cohomology, we obtain a functor:
$$H^*_b(F_{\bullet}; \calA_\bullet) \colon (\Delta \shortdownarrow Y) \to \mathrm{gr} \Mod^{\sn}_R.$$

\begin{prop}\label{prop:all:fibers:are:Kan:and:we}
Let $f \colon X \to Y$ be a Kan fibration and let~$\calA$ be a seminormed local coefficient system on~$X$. 
Then the following hold:
\begin{enumerate}[label=\enum]
    \item for every $n \in \N$ and every $0$-simplex $\tau \in (Y^{\Delta^n})_0$, the fiber $F_\tau$ is a Kan complex;
\item for every morphism~$\alpha\colon \tau\to \tau'$ in~$\Delta\shortdownarrow Y$, the induced map~$F(\alpha)\colon (F_{\tau'};\calA_{\tau'})\to (F_\tau;\calA_\tau)$ is a homotopy equivalence in~$\calL^*$.
\end{enumerate}
\end{prop}
\begin{proof}
(i) Since $f\colon X\to Y$ is a Kan fibration, the induced map $f^{\Delta^n} \colon X^{\Delta^n} \to Y^{\Delta^n}$ is also a Kan fibration~\cite[Corollary~I.5.3]{goerss2009simplicial}. Hence its fibers are Kan complexes. 

(ii) 
The induced map between the fibers $F_{\tau'}\to F_\tau$ is a weak equivalence of Kan complexes~\cite[Corollary II.8.13]{goerss2009simplicial}. 
Since we already know that $F(\alpha)^*\calA_\tau \Rightarrow \calA_{\tau'}$ is a natural isomorphism, the map~$F(\alpha)$ is a homotopy equivalence in $\calL^*$~by Proposition \ref{prop:characterisation:simpl:homotopy:eq}. 
\end{proof}

Given a Kan fibration $f \colon X \to Y$ between simplicial sets, every path~$\gamma \in (Y^{\Delta^1})_0$ with~$\gamma(0)=y$ and~$\gamma(1)=z$ induces a zigzag of homotopy equivalences in~$\calL^*$
\[
	(F_y;\calA_y)\leftarrow (F_\gamma;\calA_\gamma) \to (F_z;\calA_z)
\]
and hence, by Corollary~\ref{cor:weak invariance bc}, an isometric isomorphism in bounded cohomology
\[
	\gamma_*\colon H^*_b(F_y;\calA_y)\xrightarrow{\cong} H^*_b(F_z;\calA_z)
\]
which depends only on the homotopy class of $\gamma$. Thus, for every~$q\geq 0$, we obtain a seminormed local coefficient system
\begin{equation}
\label{eqn:coefficient fibres}
	\calH^q_b(F;\calA)\colon \pi(Y)\to \Mod^\sn_R, \quad y\mapsto H^q_b(F_y;\calA_y).
\end{equation}
If the fiber~$F_y$ is connected, then we have the following identification in degree~$0$ (Proposition~\ref{prop:degree zero}) for every basepoint~$x_y\in (F_y)_0$:
\[
	H^0_b(F_y;\calA_y)\cong \calA_y(x_y)^{\pi_1(F_y,x_y)}.
\]

\section{Serre spectral sequence in bounded cohomology}\label{sec:sss}
	Let~$f\colon X\to Y$ be a Kan fibration of simplicial sets
    and let~$\calA$ be a seminormed local coefficient system on~$X$ (over a normed ring $R$).
We denote by~$\iota_y\colon F_y\coloneqq \Delta^0 \times_{Y} X \hookrightarrow X$ the fiber inclusion over a $0$-simplex $y \in Y_0$, and by~$\calA_y\coloneqq \iota_y^*\calA$ the restricted coefficient system on~$F_y$.
	We obtain a seminormed local coefficient system~\eqref{eqn:coefficient fibres} on~$Y$ for every~$q \geq 0$:
	\[
		\calH^q_b(F;\calA)\colon \pi(Y)\to \Mod^\sn_R, \quad y\mapsto H^q_b(F_y;\calA_y).
	\]	
In this section we prove our main result, the existence of a Serre spectral sequence in bounded cohomology:

\begin{thm}[Serre spectral sequence in bounded cohomology]
\label{thm:sss}
	Let~$f\colon X\to Y$ be a Kan fibration of simplicial sets and let~$\calA$ be a seminormed local coefficient system on~$X$.
	Then there is a first-quadrant spectral sequence~$(E^{\bullet,\bullet}_r)_r$ converging to $H^*_b(X;\calA)$.
    We have an isomorphism of seminormed $R$-modules
    \[
        E^{p,q}_2\cong H^p_b(Y;\calH^q_b(F;\calA))
    \]
    in each of the following cases:
    \begin{enumerate}[label=\enum]
        \item $Y$ has only finitely many $n$-simplices in every dimension $n \leq p+1$; 
        \item $q=0$;
        \item $Y$ is connected and for some 0-simplex~$y\in Y_0$ the seminormed $R$-cochain complex~$C^\bullet_b(F_y;\calA_y)$ satisfies~$\UBC^q$.
    \end{enumerate}
\end{thm}

\subsection{Bisimplicial set of a map}

We follow the construction of the classical Serre spectral sequence due to Dress~\cite{dress}. This applies to an arbitrary map $f \colon X \to Y$ which is not necessarily a Kan fibration.

\begin{defi}[Bisimplicial set of a map]
Let~$f\colon X\to Y$ be a map of simplicial sets.
A \emph{$(p,q)$-simplex of~$f$} is a diagram of the form
\[\begin{tikzcd}
	\Delta^p\times \Delta^q \ar{r}{\sigma} \ar{d}{\pr_1} & X \ar{d}{f} \\
	\Delta^p \ar{r}{\tau} & Y
\end{tikzcd}\]
where $\pr_1\colon \Delta^p\times \Delta^q\to \Delta^p$ is the simplicial projection onto the first factor.
We denote by~$S_{p,q}f$ the set of all $(p,q)$-simplices of~$f$. By identifying $n$-simplices with maps from $\Delta^n$, this can equivalently be expressed by
\[
	S_{p,q}f\coloneqq \bigl\{(\sigma,\tau)\in (X^{\Delta^q})_p \times Y_p \bigm\vert f^{\Delta^q}(\sigma) = c(\tau) \bigr\}
\]
where $c \colon Y \to Y^{\Delta^q}$ denotes the inclusion of constant maps induced by $\Delta^q \to \Delta^0$.
The \emph{bisimplicial set~$S_\bb f$ of~$f$} is the functor
\[
	S_\bb f\colon \Delta^\op\times \Delta^\op\to \Set,
\]
which is given on objects by $([p],[q])\mapsto S_{p,q}f$ and on morphisms by precomposition.
\end{defi}

Such a bisimplicial set can be interpreted as a simplicial object in the category of simplicial sets in two ways; by fixing~$q$ and letting~$p$ vary, and vice versa.
First, we regard $S_\bb f$ as the functor (using the same notation):
\[
	S_\bb f\colon \Delta^\op\to \SSet, \quad[q]\mapsto S_{\bullet,q}f.
\]
 A $(p,q)$-simplex of~$f$ is equivalently specified by an adjoint diagram of the form
\[\begin{tikzcd}
	\Delta^p \ar{r} \ar{d} & X^{\Delta^q} \ar{d}{f^{\Delta^q}} \\
	Y \ar{r}{c} & Y^{\Delta^q}.
\end{tikzcd}\]
Such a diagram is in turn equivalent to a $p$-simplex
 $\Delta^p\to Y\times_{Y^{\Delta^q}} X^{\Delta^q}$.
 Hence we may identify~$S_\bb f$ with the functor
\begin{equation}
\label{eqn:X_bulletbullet}
    \Delta^\op\to \SSet, \quad [q]\mapsto X_{\bullet,q}\coloneqq (Y\times_{Y^{\Delta^q}} X^{\Delta^q})_\bullet.
\end{equation}

We claim that for every map $[q] \to [q']$ in~$\Delta$, the induced map $X_{\bullet, q'} \to X_{\bullet, q}$ is a homotopy equivalence.
Indeed, the map~$[q]\to [0]$ in~$\Delta$ induces a map $X \cong X_{\bullet, 0} \to X_{\bullet, q}$ which admits a retraction 
\[
    \rho_q\colon X_{\bullet, q} \to X_{\bullet, 0}
\]
induced by the map $0\colon [0] \to [q]$. The homotopy $h \colon \Delta^q \times \Delta^1 \to \Delta^q$ from the constant map at the $0$-th vertex to the identity map 
induces a map
$$X_{\bullet, q} = Y \times_{Y^{\Delta^q}} X^{\Delta^q} \to  Y^{\Delta^1} \times_{Y^{\Delta^q \times \Delta^1}} X^{\Delta^q \times \Delta^1} = X_{\bullet, q}^{\Delta^1}$$
defined by $h$ and the constant paths $Y \to Y^{\Delta^1}$. This provides a homotopy from the composite $X_{\bullet, q} \to X_{\bullet, 0} \to X_{\bullet, q}$ to the identity map which shows that~$\rho_q$ is a homotopy equivalence.
Hence, for any map~$\alpha\colon [q]\to [q']$ in~$\Delta$, we obtain a triangle of homotopy equivalences
\[
	\begin{tikzcd}
		X_{\bullet,q'} \ar{rr}{\alpha^*}[swap]{\simeq}\ar{dr}{\simeq}[swap]{\rho_{q'}}
		&& X_{\bullet,q}\ar{dl}{\rho_q}[swap]{\simeq} \\
		& X_{\bullet, 0}
	\end{tikzcd}
\]
which commutes up to canonical homotopy. In addition, it follows that any two maps $\alpha, \alpha' \colon [q] \to [q']$ induce homotopic maps $\alpha^*, (\alpha')^*$. 

	In the above situation, let additionally a seminormed local coefficient system~$\calA$ on~$X \cong X_{\bullet, 0}$ 
    be given.
	Denote by $\calA_q\coloneqq \rho_q^*\calA$ the restricted coefficient system on~$X_{\bullet, q}$.
	We obtain a refinement of~$X_{\bullet, q}$ with values in~$\calL^*$:
	\[
		(X_{\bullet, q};\calA_\bullet)\colon \Delta^\op\to \calL^*, \quad [q]\mapsto (X_{\bullet, q};\calA_q)
	\]
where the morphisms of $\Delta^\op$ are sent to homotopy equivalences in $\calL^*$. Indeed, the homotopy constructed above extends to a homotopy in $\calL^*$ that is defined similarly and is constant at the level of local coefficients -- because the homotopy $h$ is constant at the $0$-th vertex. 

\medskip
The second interpretation of~$S_\bb f$ is as the functor (using again the same notation):
\[
	S_\bb f\colon \Delta^\op\to \SSet, \quad [p]\mapsto S_{p,\bullet}f.
\]
A $(p,q)$-simplex~$(\sigma,\tau)$ of~$f$ is equivalently specified by an adjoint diagram of the form
\[\begin{tikzcd}
	\Delta^q \ar{r} \ar{d} & X^{\Delta^p} \ar{d}{f^{\Delta^p}} \\
	\Delta^0 \ar{r}{\tau} & Y^{\Delta^p}
\end{tikzcd}\]
which is in turn equivalent to a $q$-simplex $\Delta^q\to (f^{\Delta^p})^{-1}(\tau)$.
For any~$\tau\in (Y^{\Delta^p})_0$, we write $F_\tau\coloneqq (f^{\Delta^p})^{-1}(\tau)$ as in Section~\ref{sec:coefficient fibres}. 
Hence we may identify~$S_\bb f$ with the functor 
\begin{equation}\label{eqn:firstqthenp}
    \Delta^\op\to \SSet, \quad [p]\mapsto \coprod_{\tau\in (Y^{\Delta^p})_0} (F_\tau)_\bullet.
\end{equation}
Given a seminormed local coefficient system on~$X$, this functor admits a refinement with values in~$\calL^*$ (Section~\ref{sec:coefficient fibres}).

\subsection{Serre spectral sequence}
Let~$f\colon X\to Y$ be a map of simplicial sets and let~$\calA$ be a seminormed local coefficient system on~$X$.
We define the \emph{double cochain complex $(C^\bb_b(f;\calA),d_h,d_v)$ of bounded $(p,q)$-cochains of~$f$ with coefficients in~$\calA$} as follows:
the modules are given by
\[
	C^{p,q}_b(f;\calA) \coloneqq \bigl\{\varphi\colon S_{p,q}f\to \prod_{x\in X_0} \calA(x) \bigm\vert \varphi(\sigma,\tau)\in \calA(\sigma(0,0)), \ \ \|\varphi\|_\infty<+\infty \bigr\}
\]
where $\|\varphi\|_\infty\coloneqq \sup_{(\sigma, \tau) \in S_{p,q}f} \|\phi(\sigma, \tau) \|_{\calA(\sigma(0, 0))}$.
The horizontal and vertical differentials
\begin{align*}
	d_h\colon & C_b^{p,q}(f;\calA)\to C_b^{p+1,q}(f;\calA) \\
	d_v\colon & C_b^{p,q}(f;\calA)\to C_b^{p,q+1}(f;\calA)
\end{align*}
are given respectively by
\begin{align*}
(d_h \varphi) (\sigma, \tau) \coloneqq\ &  \mathcal{A}([\sigma |_{[(0, 0), (1, 0)]}]^{-1})\varphi (\sigma\circ (\delta_0^{p+1} \times \id_{\Delta^q}), \tau \circ \delta_0^{p+1}) \\
&+ \sum_{i = 1}^{p+1} (-1)^i \varphi (\sigma\circ (\delta_i^{p+1} \times \id_{\Delta^q}), \tau \circ \delta_i^{p+1}) \\
(d_v \varphi) (\sigma, \tau) \coloneqq\ & (-1)^p \mathcal{A}([\sigma |_{[(0, 0), (0, 1)]}]^{-1}) \varphi(\sigma\circ (\id_{\Delta^p} \times \delta^{q+1}_0), \tau) \\
&+\sum_{j = 1}^{q+1} (-1)^{j+p} \varphi(\sigma \circ (\id_{\Delta^p} \times \delta^{q+1}_j), \tau).
\end{align*}
Here $\delta^{*}_k \colon \Delta^{*-1} \to \Delta^{*}$ is the inclusion of the $k$-th face. 
It is straight-forward to verify that $(C^{\bullet,\bullet}_b(f;\calA),d_h,d_v)$ is indeed a double cochain complex, that is, we have $d_h \circ d_h = 0$, $d_v \circ d_v = 0$, and $d_v \circ d_h + d_h \circ d_v = 0$.
Note that the differentials of the double cochain complex are induced from the simplicial operators of the bisimplicial set $S_{\bullet, \bullet} f$ in the standard way (up to a convention for the signs). 

The double cochain complex~$C^\bb_b(f;\calA)$ yields two spectral sequences~$(\firstE^\bb_r)_r$ and~$(\secondE^\bb_r)_r$ with second pages
\begin{align*}
	\firstE^{q,p}_2 &\cong H^q(H^p(C^\bb_b(f;\calA),d_h),d_v) \\
	\secondE^{p,q}_2 &\cong H^p(H^q(C^\bb_b(f;\calA),d_v),d_h)
\end{align*}
both converging to the cohomology of the associated total complex $\Tot(C^\bb_b(f;\calA))$.

First, we examine the spectral sequence~$(\firstE^\bb_r)_r$ which allows us to identify the cohomology of the total complex. We emphasize that the map~$f$ need not be a Kan fibration.

\begin{lemma}
\label{lem:firstE}
	Let~$f\colon X\to Y$ be a map of simplicial sets and let~$\calA$ be a seminormed local coefficient system on~$X$. We have an isomorphism of seminormed $R$-modules
	\[
		\firstE^{q,p}_2\cong \begin{cases}
			H^p_b(X;\calA) & \textup{if }q=0; \\
			0 & \textup{if }q>0.
		\end{cases}
	\]
	In particular, we have $H^*(\Tot(C^\bb_b(f;\calA)))\cong H^*_b(X;\calA)$ (as $R$-modules).
	\begin{proof}
		Under the identification of~$S_{\bullet,\bullet}f$ as in~\eqref{eqn:X_bulletbullet}, the first page takes the form
        \[
            \firstE^{q,p}_1 = H^p_b(X_{\bullet,q};\calA_q).
        \]
		In order to identify the second page, we have to explicitly describe the differential $\firstd_1^{q,p}\colon \firstE^{q,p}_1\to \firstE^{q+1,p}_1$.
		Every map~$[q]\to [q+1]$ in~$\Delta$ induces a homotopy equivalence $(X_{\bullet, q+1};\calA_{q+1})\to (X_{\bullet, q};\calA_q)$ in~$\calL^*$ and hence an (isometric) isomorphism on bounded cohomology (Corollary~\ref{cor:homotopic:map:same:map:in:bc})
		\[
			H^p_b(X_{\bullet, q};\calA_{q})\xrightarrow{\cong} H^p_b(X_{\bullet, q+1};\calA_{q+1}).
		\]
		Since any two maps from~$\Delta^{q}$ to~$\Delta^{q+1}$ are canonically homotopic, they induce the same isomorphism on bounded cohomology by Corollary~\ref{cor:homotopic:map:same:map:in:bc}.
		Identifying these cohomology groups with~$H^p_b(X;\calA) \cong H^p_b(X_{\bullet, 0};\calA)$,
		the differential becomes the map
		\[
			\firstd^{q,p}_1\colon H^p_b(X;\calA)\to H^p_b(X;\calA)
		\]
		that is the identity for odd~$q$ and zero for even~$q$.
		Hence the second page~$\firstE^\bb_2$ can be identified as claimed.
		In particular, the spectral sequence~$(\firstE^\bb_r)_r$ collapses at the second page and converges to~$H^*_b(X;\calA)$.
	\end{proof}
\end{lemma}

Second, we examine the spectral sequence~$(\secondE^\bb_r)_r$.
In order to obtain a useful identification of the second page, we assume that~$f$ is a Kan fibration and that the collection of bounded cochain complexes of the fibers satisfies the bounded boundary condition (Definition~\ref{defn:Phi iso}).

\begin{lemma}
\label{lem:secondE}
    Let~$f\colon X\to Y$ be a Kan fibration of simplicial sets and let~$\calA$ be a seminormed local coefficient system on~$X$.
    Let~$p,q\in \N$.
    If the collections  $(C^\bullet_b(F_\tau;\calA_\tau))_{\tau\in (Y^{\Delta^p})_0}$
    and $(C^\bullet_b(F_\tau;\calA_\tau))_{\tau\in (Y^{\Delta^{p+1}})_0}$
    satisfy~$\BBC^q$, then we have an isomorphism of seminormed $R$-modules
    \begin{equation}
    \label{eqn:secondE identif}
        \secondE^{p,q}_2\cong H^p_b(Y;\calH^q_b(F;\calA)).
    \end{equation}
    In particular, we have an isomorphism~\eqref{eqn:secondE identif} in each of the following cases:
    \begin{enumerate}[label=\enum]
        \item $Y$ has only finitely many $n$-simplices in every dimension~$n\le p+1$;
        \item $q=0$;
        \item $Y$ is connected and for some 0-simplex~$y\in Y_0$ the seminormed $R$-cochain complex~$C^\bullet_b(F_y;\calA_y)$ satisfies~$\UBC^q$.
    \end{enumerate}
    \begin{proof}
        Under the identification of~$S_{\bullet,\bullet}f$ as in~\eqref{eqn:firstqthenp}, the first page takes the form
        \[
			\secondE^{p,q}_1 = H^q\Bigl(\bprod{\tau\in (Y^{\Delta^p})_0} C^\bullet_b(F_\tau;\calA_\tau)\Bigr).
		\]
        Since the collection~$(C^\bullet_b(F_\tau,\calA_\tau))_{\tau\in (Y^{\Delta^p})_0}$ satisfies~$\BBC^q$ by assumption, 
        we obtain (as seminormed $R$-modules)
        \[
            \secondE^{p,q}_1\cong \bprod{\tau\in (Y^{\Delta^p})_0}H^q_b(F_\tau;\calA_\tau).
        \]
        Moreover, since~$f$ is a Kan fibration, the map $(F_{\tau};\calA_{\tau})\to (F_{\tau_0};\calA_{\tau_0})$ induced by $\tau_0 \to \tau$ is a homotopy equivalence in~$\calL^*$ (Proposition~~\ref{prop:all:fibers:are:Kan:and:we}). 
        Hence we have (as seminormed $R$-modules)
        \[
			\secondE^{p,q}_1\cong \bprod{\tau\in (Y^{\Delta^p})_0}H^q_b(F_{\tau_0};\calA_{\tau_0})
			=C^p_b(Y;\calH^q_b(F;\calA))
		\]
        and similarly for~$\secondE_1^{p+1,q}$.
        By inspection, the differential $\secondd_1^{p,q}\colon \secondE_1^{p,q}\to \secondE_1^{p+1,q}$ agrees with the $p$-th differential of~$C^\bullet_b(Y;\calH^q_b(F;\calA))$.
        We have a commutative diagram
        \[\begin{tikzcd}
            \secondE_1^{p-1,q}\ar{r}{\secondd_1^{p-1,q}}\ar[two heads]{d}
            & \secondE_1^{p,q}\ar{r}{\secondd_1^{p,q}}\ar{d}{\cong}
            & \secondE_1^{p+1,q}\ar{d}{\cong}
            \\
            C_b^{p-1}(Y;\calH_b^q(F;\calA))\ar{r}{d^{p-1}}
            & C_b^p(Y;\calH_b^q(F;\calA))\ar{r}{d^p}
            & C_b^{p+1}(Y;\calH_b^q(F;\calA))
        \end{tikzcd}\]
		Thus the second page can be identified as claimed (as seminormed $R$-modules):
		\[
			\secondE^{p,q}_2\cong H^p_b(Y;\calH^q_b(F;\calA)).
		\]        
        It remains to show that the collection  $(C^\bullet_b(F_\tau;\calA_\tau))_{\tau\in (Y^{\Delta^k})_0}$
        for~$k\in \{p,p+1\}$
        satisfies~$\BBC^q$ in each of the cases~(i), (ii), and~(iii).
        In case~(i), this holds because finite collections satisfy~$\BBC^q$ (Example~\ref{ex:finite collections}).
        In case~(ii), each seminormed $R$-cochain complex~$C^\bullet_b(F_\tau;\calA_\tau)$ clearly satisfies 0-$\UBC^0$.
        Hence the collection~$(C^\bullet_b(F_\tau;\calA_\tau))_{\tau\in (Y^{\Delta^k})_0}$ satisfies~$\BBC^0$ by Proposition~\ref{prop:bounded products}.
        In case~(iii), since~$f$ is a Kan fibration and~$Y$ is connected, $(F_y;\calA_y)$ is homotopy equivalent to~$(F_\tau;\calA_\tau)$ in~$\calL^*$ for every~$\tau\in (Y^{\Delta^k})_0$ (Proposition~\ref{prop:all:fibers:are:Kan:and:we}).
        Since the seminormed $R$-cochain complex~$C^\bullet_b(F_y;\calA_y)$ satisfies~$\UBC^q$, the collection~$(C^\bullet_b(F_\tau;\calA_\tau))_{\tau\in (Y^{\Delta^k})_0}$ satisfies~$\BBC^q$ by Corollary~\ref{cor:weak invariance UBC}.
    \end{proof}
\end{lemma}

\begin{proof}[Proof of Theorem~\ref{thm:sss}]  
	This now follows from Lemma~\ref{lem:firstE} and Lemma~\ref{lem:secondE}.
\end{proof}

\begin{rem}
    It is not strictly necessary to restrict ourselves to Kan fibrations in Lemma~\ref{lem:secondE} and Theorem~\ref{thm:sss}.
    In order to treat arbitrary simplicial maps, one has to allow for more general coefficient systems over the simplex category (Definition~\ref{defn:simplex category}).
    Given a simplicial set~$Y$ and a functor $\calB\colon \Delta\shortdownarrow Y\to \Mod^\sn_R$, one defines the bounded cohomology $H^*_b(Y;\calB)$ similarly.

    Then Lemma~\ref{lem:secondE} can be extended to arbitrary maps as follows. Let $f\colon X\to Y$ be a map of simplicial sets and let~$\calA$ be a seminormed local coefficient system on~$X$.
    Consider the functor
	\[
		\calH^q_b(F_\bullet;\calA_\bullet)\colon \Delta\shortdownarrow Y\to \Mod^\sn_R, \quad \tau\mapsto H^q_b(F_\tau;\calA_\tau).
	\]
	If the collections of seminormed $R$-cochain complexes~$(C^\bullet_b(F_\tau;\calA_\tau))_{\tau\in (Y^{\Delta^p})_0}$
    and $(C^\bullet_b(F_\tau;\calA_\tau))_{\tau\in (Y^{\Delta^{p+1}})_0}$
    satisfy~$\BBC^q$, then we can identify (as seminormed $R$-modules)
        \[
		\secondE^{p,q}_2 \cong H^p_b(Y;\calH^q_b(F_\bullet;\calA_\bullet)).
	\]
    Hence one obtains a corresponding version of Theorem~\ref{thm:sss} for arbitrary simplicial maps.
\end{rem}

\begin{rem}\label{rem:replacement:fibration}
An alternative way to extend the spectral sequence to arbitrary maps is by considering homotopy fibers. Given a simplicial map $f \colon X \to Y$ (resp.\ between Kan complexes), it is possible to replace~$f$ by a Kan fibration, i.e., $f = p \, \circ \, i$, where $i \colon X \to Z$ is a weak equivalence (resp.\ between Kan complexes) and $p \colon Z \to Y$ a Kan fibration. 
The fiber(s) of~$p$ is the homotopy fiber(s) of~$f$.
Then one can study the Serre spectral sequence in bounded cohomology associated to $p \colon Z \to Y$. Assuming that $X$ and~$Y$ are Kan complexes and given a seminormed local coefficient system~$\calA$ on $Z$, then by homotopy invariance of bounded cohomology, it follows that we have an isometric isomorphism $H_b^*(Z; \calA) \cong H_b^*(X; i^*\calA)$. 
\end{rem}

As a special case, for trivially seminormed local coefficients, Theorem~\ref{thm:sss} recovers the Serre spectral sequence in ordinary cohomology.
Recall from Section~\ref{sec:bc} that for a seminormed local coefficient system~$\calA$, we denote by~$\calA_{\un}$ the underlying local coefficient system (after forgetting the seminorm) and by~$\calA_{\tr}$ the corresponding seminormed local coefficient system equipped with the trivial seminorm.

\begin{example}
\label{ex:ordinary ss}
	Let~$f\colon X\to Y$ be a Kan fibration of simplicial sets. Let~$\calA$ be a seminormed local coefficient system on~$X$.
	Then, for every~$p\in \N$ and every $p$-simplex~$\tau\colon \Delta^p\to Y$, the bounded cochain complex $C^\bullet_b(F_\tau;(\calA_{\tr})_\tau)$ coincides with the ordinary cochain complex $C^\bullet(F_\tau;(\calA_{\un})_\tau)$, which clearly satisfies the uniform boundary condition in all degrees.
    Then the collection $(C^\bullet_b(F_\tau;(\calA_{\tr})_\tau))_{\tau\in (Y^{\Delta^p })_0}$ satisfies the bounded boundary condition in all degrees.
	Hence, using Lemma~\ref{lem:secondE}, the spectral sequence in Theorem~\ref{thm:sss} takes the usual form
	\[
		H^p(Y;\calH^q(F;\calA_{\un})) \Rightarrow H^{p+q}(X;\calA_{\un}).
	\] 
\end{example}

\begin{rem}
\label{rem:non-BBC}
The bounded boundary condition is required in general for the identification of the $E_2$-page in Theorem \ref{thm:sss} (also for $p=0$). For an elementary example, consider a collection of Kan complexes $(X_i)_{i \in I}$ such that the corresponding collection of bounded cochain complexes (with constant coefficients in $\R$) does not satisfy $\BBC^q$ for some $q \geq 0$. The coproduct $\coprod_{i\in I} X_i$ defines a Kan fibration over the index set $I$ regarded as a discrete simplicial set. Then we have
$$H^*_b\bigl(\coprod_{i\in I} X_i ; \R\bigr) = H^*\bigl(\bprod{i\in I} C^{\bullet}_b(X_i; \R)\bigr)$$
and 
$$\secondE^{p, *}_1 \cong H^*\bigl(\bprod{i\in I} C^{\bullet}_b(X_i; \R)\bigr)$$
for every $p \geq 0$. Moreover, the differential $\secondd_1^{p,q}$ is the identity if $p$ is odd and zero if $p$ is even. Thus, the spectral sequence collapses at the $E_2$-page, but its term $E^{0,q}_2$ is not identified in this case with $\bprod{i\in I} H^q_b(X_i; \R)$.  
\end{rem}

\subsection{Functoriality} We discuss the functoriality of the Serre spectral sequence in bounded cohomology (Theorem~\ref{thm:sss}).
Consider a commutative square in~$\SSet$
\[\begin{tikzcd}
    X\ar{r}{g}\ar{d}{f} & X'\ar{d}{f'} \\
    Y\ar{r}{h} & Y'
\end{tikzcd}\]
where~$f$ and~$f'$ are Kan fibrations, together with a map $(g,\theta)\colon (X;\calA)\to (X';\calA')$ in~$\calL^*$.
This induces a natural map of double complexes 
$$C^\bb_b(g, \theta, h) \colon C^\bb_b(f';\calA')\to C^\bb_b(f;\calA)$$ 
and hence a map of spectral sequences.

In particular, a specific change of coefficients induces a map of Serre spectral sequences from bounded cohomology to ordinary cohomology:

\begin{prop}[Comparison map of spectral sequences]
\label{prop:comparison ss}
	Let~$f\colon X\to Y$ be a Kan fibration of simplicial sets and let~$\calA$ be a seminormed local coefficient system on~$X$.
	Let~$(E^{\bullet,\bullet}_r)_r$ be the Serre spectral sequence in bounded cohomology (Theorem~\ref{thm:sss}) and let~$(D^{\bullet,\bullet}_r)_r$ be the Serre spectral sequence in ordinary cohomology (Example~\ref{ex:ordinary ss}).
	Then there is a map of spectral sequences $(c_r\colon E^{\bullet,\bullet}_r\to D^{\bullet,\bullet}_r)_r$ such that
	\[\begin{tikzcd}
		E^{p,q}_r\ar[Rightarrow]{r}\ar{d}{c_r} & H^{p+q}_b(X;\calA)\ar{d}{\comp} \\
		D^{p,q}_r\ar[Rightarrow]{r} & H^{p+q}(X;\calA_{\un}).
	\end{tikzcd}\]
	Moreover, for all~$p,q\in \N$, we have a factorization
	\[\begin{tikzcd}
		E^{p,q}_2\ar{rr}{c_2}\ar{dr} && D^{p,q}_2\cong H^p(Y;\calH^q(F;\calA_{\un})) \\
		& H^p_b(Y;\calH^q_b(F;\calA))\ar{ur}	
	\end{tikzcd}\] 
	\begin{proof}
        We consider the square in~$\SSet$
        \[\begin{tikzcd}
            X\ar{r}{\id_X}\ar{d}{f} & X\ar{d}{f} \\
            Y\ar{r}{\id_Y} & Y
        \end{tikzcd}\]
        together with the map $(\id_X,\eta(\calA))\colon (X;\calA_{\tr})\to (X;\calA)$, where~$\eta$ is the unit map from~\eqref{eqn:unit}.
        This induces the desired map~$c$ of spectral sequences.
		Naturality of the map~$\Phi$ from~\eqref{comparison-map-Phi} yields a commutative square of $R$-modules
		\[\begin{tikzcd}
			E^{p,q}_1\cong H^q(\prod^b_{\tau\in (Y^{\Delta^p})_0}C^{\bullet}_b(F_\tau;\calA_\tau)) \ar{r}{c_1}\ar{d}[swap]{\Phi}
			& D^{p,q}_1\cong H^q(\prod_{\tau\in (Y^{\Delta^p})_0}C^{\bullet}(F_\tau;(\calA_{\un})_\tau)) \ar{d}{\cong}[swap]{\Phi} \\
			\prod^b_{\tau\in (Y^{\Delta^p})_0}H^q(C^{\bullet}_b(F_\tau;\calA_\tau)) \ar{r} 
			& \prod_{\tau\in (Y^{\Delta^p})_0}H^q(C^{\bullet}(F_\tau;(\calA_{\un})_\tau)) 
			\end{tikzcd}\]
		which is natural in $p$ and the right vertical map is an isomorphism.
		Taking cohomology~$H^p$ in the $p$-direction yields the desired factorization of~$c_2$.
	\end{proof}
\end{prop}


\section{Serre spectral sequence in \texorpdfstring{$\ell^1$}{l\textonesuperior}-homology}\label{sec:l1-homology}

There is a corresponding Serre spectral sequence in $\ell^1$-homology.
We keep the exposition brief and omit the proofs, which are analogous (dual) to the case of bounded cohomology.

\begin{defi}[$\ell^1$-direct sum]
    Let~$(C^i)_{i\in I}$ be a collection of seminormed $R$-modules.
    The \emph{$\ell^1$-direct sum~$\lsum{i\in I}C^i$} is the $R$-submodule of~$\prod_{i\in I}C^i$ consisting of the elements~$(c_i)_{i\in I}$ such that $c_i=0$ for all but countably many~$i\in I$ and
    \[
        \sum_{i\in I}\|c_i\|_{C^i}<+\infty.
    \]
    The $R$-module~$\lsum{i\in I}C^i$ is equipped with the $\ell^1$-seminorm.
\end{defi}

\begin{defi}[$\ell^1$-homology with local coefficients]\label{rem:def:ell1:local:coeff}
Let $X$ be a simplicial set and let $\calA$ be a seminormed local coefficient system on $X$.
The \emph{$\ell^1$-chain complex $C_\bullet^{\ell^1}(X; \calA)$ of $X$ with local coefficients in $\calA$} consists of the seminormed $R$-modules
\[
C_n^{\ell^1}(X; \mathcal{A}) \coloneqq \lsum{\sigma \in X_n} \mathcal{A}(\sigma_0) \cdot \sigma,
\]
together with the boundary operators
\[
    \partial_n \colon C_n^{\ell^1}(X; \calA) \to C_{n-1}^{\ell^1}(X; \mathcal{A})
\]
given coordinate-wise by
\[
    \partial_n(a \cdot \sigma)= \mathcal{A}(\sigma_{[0, 1]})(a) \cdot \partial_0 \sigma + \sum_{i = 1}^n (-1)^i \cdot a \cdot \partial_i \sigma.
\]
The \emph{$\ell^1$-homology of $X$ with local coefficients in $\calA$} is then defined by 
\[
H^{\ell^1}_*(X; \calA) \coloneqq H_*(C_\bullet^{\ell^1}(X; \calA), \partial_\bullet).
\]
\end{defi}

We define an associated (covariant) category~$\mathcal{L}$ as follows: 
\begin{itemize}
\item the objects are pairs~$(X;\calA)$, where~$X$ is a simplicial set and~$\calA$ is a seminormed local coefficient system on~$X$;
\item a morphism $(X;\calA) \to (Y; \calB)$ is a pair~$(f, \theta)$, where $f \colon X \to Y$
is a map of simplicial sets and $\theta \colon \calA \Rightarrow f^*\calB$ is a natural transformation of seminormed local coefficient systems on $X$.
\end{itemize}
Then $\ell^1$-homology defines a functor from~$\calL$ to the category of graded seminormed $R$-modules:
\[
H_*^{\ell^1} \colon \calL \to \mathrm{gr}\Mod^\sn_R.
\]
A homotopy of maps in~$\calL$ is defined analogously to a homotopy in~$\calL^*$ and $\ell^1$-homology is homotopy invariant:

\begin{lemma}[cf.~ Lemma~\ref{lemma:simpl:hom:maps:induce:bounded:chain:} and Corollary~\ref{cor:homotopic:map:same:map:in:bc}]
    Let~$(f,\theta),(f',\theta')\colon (X;\calA)\to (Y;\calB)$ be maps in~$\calL$. 
    If there exists a homotopy from~$(f,\theta)$ to~$(f',\theta')$, then the chain maps $C_\bullet^{\ell^1}(f, \theta)$ and $C_\bullet^{\ell^1}(f', \theta')$ are chain homotopic via a chain homotopy whose norm in degree~$n$ is bounded by $n+1$ for every~$n\ge 0$.

    In particular, two homotopic maps in~$\calL$ induce the same map in $\ell^1$-homology and homotopy equivalences in $\calL$ induce isometric isomorphisms in $\ell^1$-homology.
\end{lemma}

\begin{cor}[cf.~Corollary~\ref{cor:weak invariance bc}]
\label{prop:weak invariance l1}
Let $(f, \theta) \colon (X; \calA) \to (Y; \calB)$ be a map in~$\calL$ such that $f$ is a weak equivalence between Kan complexes and $\theta$ is a natural isomorphism. Then $(f,\theta)$ induces an isometric isomorphism in $\ell^1$-homology. 
\end{cor}

The interaction between homology and $\ell^1$-direct sums is crucial.

\begin{defi}[cf.~Definition~\ref{defn:Phi iso}]
    Let $((C^i_\bullet,\partial^i_\bullet))_{i\in I}$ be a uniform collection of seminormed $R$-chain complexes.
    The \emph{$\ell^1$-direct sum}~$\lsum{i\in I}C^i_\bullet$ is defined degreewise.
    We say that the uniform collection~$((C^i_\bullet,\partial^i_\bullet))_{i\in I}$ satisfies the \emph{bounded boundary condition in degree~$n$} ($\BBC_n$) if the canonical map of seminormed $R$-modules
        \[
        \Phi_n \colon H_n\Big(\lsum{i \in I} C_\bullet^i\Big)\to \lsum{i \in I} H_n(C_\bullet^i)
        \]
        is an isomorphism.
\end{defi}

The uniform boundary condition ($\UBC_n$) for chain complexes is completely analogous to Definition~\ref{defn:UBC}, obtained simply by re-indexing.

\begin{cor}[cf.~Corollary~\ref{cor:control_ubc_constant}]
    Let $(f, \theta) \colon (X; \calA) \to (Y; \calB)$ be a homotopy equivalence in~$\calL$, where~$Y$ is a Kan complex.
    Let~$n\in \N$ and let~$\kappa\in \R_{\ge 0}$.
        If~$C^{\ell^1}_\bullet(X;\calA)$ satisfies $\kappa$-$\UBC_n$, then~$C^{\ell^1}_\bullet(Y;\calB)$ satisfies $(\kappa+n+1)$-$\UBC_n$.
\end{cor}

Using a version of Proposition~\ref{prop:bounded products} for chain complexes, we obtain:

\begin{cor}
[cf.~Corollary~\ref{cor:weak invariance UBC}]
Let~$((X_i;\calA_i))_{i\in I}$ be a collection of objects in~$\calL$, whose underlying simplicial sets are Kan complexes, that are pairwise  homotopy equivalent in $\calL$. 
 If for some~$i_0\in I$, the seminormed $R$-chain complex~$C_\bullet^{\ell^1}(X_{i_0};\calA_{i_0})$ satisfies~$\UBC_n$, 
then the uniform collection of seminormed $R$-chain complexes \linebreak $(C_\bullet^{\ell^1}(X_i;\calA_i))_{i\in I}$ satisfies~$\BBC_n$.
\end{cor}

Let~$f\colon X\to Y$ be a Kan fibration of simplicial sets
    and let~$\calA$ be a seminormed local coefficient system on~$X$.
    For a 0-simplex~$y\in Y_0$ we denote by~$\iota_y\colon F_y \hookrightarrow X$ the fiber inclusion and by~$\calA_y\coloneqq \iota_y^*\calA$ the restricted coefficient system on~$F_y$.
    We obtain a seminormed local coefficient system on~$Y$ for every~$q\ge 0$:
	\[
		\calH_q^{\ell^1}(F;\calA)\colon \pi(Y)\to \Mod^\sn_R, \quad y\mapsto H_q^{\ell^1}(F_y;\calA_y).
	\]	

\begin{thm}[Serre spectral sequence in $\ell^1$-homology]
\label{thm:sss:l1}
	Let~$f\colon X\to Y$ be a Kan fibration of simplicial sets and let~$\calA$ be a seminormed local coefficient system on~$X$.
	Then there is a first-quadrant spectral sequence~$(E_{\bullet,\bullet}^r)_r$ converging to $H_*^{\ell^1}(X;\calA)$.
    We have an isomorphism of seminormed $R$-modules
    \[
        E^2_{p,q}\cong H^{\ell^1}_p(Y;\calH^{\ell^1}_q(F;\calA))
    \]
    in each of the following cases:
    \begin{enumerate}[label=\enum]
        \item $Y$ has only finitely many $n$-simplices in every dimension $n \leq p$;
        \item $Y$ is connected and for some 0-simplex~$y\in Y_0$ the seminormed $R$-chain complex~$C^{\ell^1}_\bullet(F_y;\calA_y)$ satisfies~$\UBC_q$.
    \end{enumerate}
\end{thm}

\begin{proof}[Sketch of the proof]
    The proof is analogous to the proof of Theorem~\ref{thm:sss} for bounded cohomology. We define the \emph{double chain complex~$(C_\bb^{\ell^1}(f;\calA),d_h,d_v)$ of $\ell^1$-$(p,q)$-chains of~$f$ with coefficients in~$\calA$} as follows:
the modules are given by
\[
	C_{p,q}^{\ell^1}(f;\calA) \coloneqq \lsum{(\sigma, \tau) \in S_{p,q}f} \calA(\sigma(0,0)) \cdot (\sigma, \tau)
\]
and the horizontal and vertical differentials
\begin{align*}
	d_h\colon & C^{\ell^1}_{p,q}(f;\calA)\to C^{\ell^1}_{p-1,q}(f;\calA) \\
	d_v\colon & C^{\ell^1}_{p,q}(f;\calA)\to C^{\ell^1}_{p,q-1}(f;\calA)
\end{align*}
are given coordinate-wise respectively by
\begin{align*}
d_h (a \cdot (\sigma, \tau)) \coloneqq\ &  \mathcal{A}([\sigma |_{[(0, 0), (1, 0)]}]) (a) \cdot (\sigma\circ (\delta_0^{p} \times \id_{\Delta^q}), \tau \circ \delta_0^{p}) \\
&+ \sum_{i = 1}^{p} (-1)^i a \cdot (\sigma\circ (\delta_i^{p} \times \id_{\Delta^q}), \tau \circ \delta_i^{p}) \\
d_v (a \cdot (\sigma, \tau)) \coloneqq\ & (-1)^p \mathcal{A}([\sigma |_{[(0, 0), (0, 1)]}])(a) \cdot (\sigma\circ (\id_{\Delta^p} \times \delta^{q}_0), \tau) \\
&+\sum_{j = 1}^{q} (-1)^{j+p} a \cdot (\sigma \circ (\id_{\Delta^p} \times \delta^{q}_j), \tau).
\end{align*}
Here $\delta^{*}_k \colon \Delta^{*-1} \to \Delta^{*}$ is the inclusion of the $k$-th face.
Then, the two interpretations~\eqref{eqn:firstqthenp} and~\eqref{eqn:secondE identif} of the bisimplicial set~$S_{\bullet, \bullet}f$ yield the corresponding identifications of the cohomology of the total complex and the second page.
\end{proof}

\section{Applications}\label{sec:applications}
In this section we present several applications of the Serre spectral sequence in bounded cohomology.

\subsection{Highly-connected maps}
An $n$-connected Kan complex has trivial bounded cohomology in degrees~$\leq n$ (Corollary~\ref{cor:n-connected bc}). Using the Serre spectral sequence, we extend this statement to $n$-connected maps -- we note that this result can also be proved more directly~\cite{Ivanov17}.
We recall that a map $f \colon X \to Y$ between Kan complexes is an $n$-\emph{equivalence}, $n\ge 0$, if $f_* \colon \pi_k(X, x) \to \pi_k(Y, f(x))$ is an isomorphism in all degrees~$k \leq n-1$ and surjective in degree $k = n$ (for all basepoints).

\begin{cor}\label{cor:n-invariance bc}
Let~$f\colon X\to Y$ be an $(n+1)$-equivalence between (based) connected Kan complexes, $n \geq 0$, and let~$\calA$ be a seminormed local coefficient system on~$X$.
Then the canonical map
\[
    H^i_b(Y;\calH^0_b(F;\calA))\to H^i_b(X;\calA)
\]
is an isomorphism (of $R$-modules) for all~$i\le n$ and injective for~$i=n+1$.
\end{cor}
\begin{proof}
By homotopy invariance of bounded cohomology, we may assume that~$f$ is a Kan fibration (Remark~\ref{rem:replacement:fibration}). 
Since the map~$f$ is an $(n+1)$-equivalence, the (based, connected) fiber~$F$ of $f$ is $n$-connected. Let $i \colon F \to X$ denote the inclusion.
Then, for all~$q\in \{1,\ldots,n\}$, we have $H^q_b(F; i^*\calA)=0$ and the seminormed $R$-cochain complex~$C^\bullet_b(F;i^*\calA)$ satisfies~$\UBC^q$ by Corollary~\ref{cor:n-connected bc}.
Thus, for~$q\in \{1,\ldots,n\}$, the terms~$E^{p,q}_2$ in the Serre spectral sequence vanish and therefore the edge homomorphisms in degrees~$\le n$ are isomorphisms (of $R$-modules). Moreover, the term~$E^{n+1,0}_2\cong H^{n+1}_b(Y;\calH^0_b(F;\calA))$ survives until the $E_\infty$-page and hence the edge homomorphism in degree~$n+1$ is injective.
\end{proof}

\subsection{Mapping theorems}

In this section, we prove Theorem~\ref{thm:mapping thm}.
Throughout this section, we work over the normed ring~$\R$.
For Banach coefficients, we obtain a generalization of Corollary~\ref{cor:n-invariance bc} assuming only that the fiber has trivial bounded cohomology:

\begin{cor}
\label{cor:acyclic fibre}
	Let~$f\colon X\to Y$ be a Kan fibration of simplicial sets, where~$Y$ is connected.
	Let~$\calA$ be a seminormed local coefficient system on~$X$ taking values in Banach $\R$-modules.
	 Let~$n\in \N$ and assume that for some 0-simplex~$y\in Y_0$ we have~$H^i_b(F_y;\calA_y)=0$ for all~$i\in \{1,\ldots,n\}$.
	 Then the canonical map
	\[
		 H^i_b(Y;\calH^0_b(F;\calA))\to H^i_b(X;\calA)
	\]
	is an isomorphism (of $\R$-modules) for all~$i\le n$ and injective for~$i=n+1$.
	\begin{proof}
		We argue that the seminormed $\R$-cochain complex~$C^\bullet_b(F_y;\calA_y)$ satisfies~$\UBC^q$ for all~$q\in \{1,\ldots,n\}$.
		Indeed, since~$\calA$ takes values in Banach modules, the cochain modules~$C^\bullet_b(F_y;\calA_y)$ are Banach.
		Since $H^i_b(F_y;\calA_y)=0$ for all~$i\in \{1,\ldots,n\}$ by assumption, the seminormed $\R$-cochain complex $C^\bullet_b(F_y;\calA_y)$ satisfies~$\UBC^q$ for all~$q\in \{1,\ldots,n\}$ by Corollary~\ref{cor:matsumoto-morita}.
        Then the claim follows from Theorem~\ref{thm:sss} as in the proof of Corollary~\ref{cor:n-invariance bc}. 
	\end{proof}
\end{cor}

We will replace the assumptions on the fiber~$F$ by assumptions on its homotopy groups.
Let~$n\ge 1$ and let~$\Gamma$ be a group (abelian if~$n\ge 2$).
A connected Kan complex~$X$ is a model for the \emph{Eilenberg--MacLane complex}~$K(\Gamma,n)$ if~$\pi_n(X)\cong \Gamma$ and~$\pi_i(X)=0$ for all~$i\neq n$.
Equivalently, the geometric realization~$|X|$ is a model for the topological Eilenberg--MacLane space.
The Kan complex~$K(\Gamma,n)$ is unique up to homotopy equivalence.
A model for~$K(\Gamma,1)$ is given by the nerve~$N(\Gamma)$ (Example~\ref{example:G-vs-BG}).

We use the following notation for bounded acyclicity with respect to a specified class of Banach coefficients. 

\begin{defi}[$\mathfrak{X}$-coefficient system]
	Let~$\mathfrak{X}$ be a subcategory of~$\Mod^\sn_\R$ consisting of Banach $\R$-modules.
	A \emph{$\mathfrak{X}$-coefficient system} on a simplicial set~$X$ is a functor~$\calA\colon \pi(X)\to \mathfrak{X}$.
	We denote
    by~$\mathfrak{X}\text{-}\BAc$ the class of groups~$\Gamma$ satisfying
    \[
        H^k_b(K(\Gamma,1);\calA)=0
    \]
    for all~$k\ge 1$ and all $\mathfrak{X}$-coefficient systems~$\calA$ on~$K(\Gamma,1)$.
    By homotopy invariance of bounded cohomology, the definition does not depend on the model for~$K(\Gamma,1)$.
\end{defi}

\begin{example}
\label{ex:classes of coeffs}
The main examples are the following:
	\begin{enumerate}[label=\enum]
		\item Let~$\mathfrak{X}$ be the full subcategory of all Banach modules.
		Then~$\mathfrak{X}$-$\BAc$ is the class of finite groups~\cite[Corollary~3.11]{Frigerio:book};
		\item Let~$\mathfrak{X}$ be the full subcategory of dual Banach modules.
		Then~$\mathfrak{X}$-$\BAc$ is the class of amenable groups~\cite[Theorem~3.12]{Frigerio:book};
		\item For the category~$\mathfrak{X}= \underline{\R}$ which consists only of $\R$ with the identity morphism, the class~$\mathfrak{X}$-$\BAc$ is that of boundedly acyclic groups by definition. The same class of groups arises if we consider instead the category of all dual Banach modules with identity morphisms (see \cite{Moraschini-Raptis2}). 
	\end{enumerate}
\end{example}

\begin{lemma}
\label{lem:Eilenberg MacLane}
    Let~$n\ge 1$ and let~$\Gamma$ be a group (abelian if $n\ge 2$).
	If~$\Gamma \in \mathfrak{X}$-$\BAc$, then we have 
	\[
		H^i_b(K(\Gamma,n);\calA)=0
	\]
	for all~$i\ge 1$ and all $\mathfrak{X}$-coefficient systems~$\calA$ on~$K(\Gamma,n)$.
	\begin{proof}
		We proceed by induction on~$n\ge 1$.
		For~$n=1$, the claim holds by the definition of~$\mathfrak{X}$-$\BAc$.
		For~$n>1$, we consider the path fibration
		\[
			K(\Gamma,n-1)\to W \xrightarrow{f} K(\Gamma,n),
		\]
  where $W$ is a contractible simplicial set. In particular, the bounded cohomology of $W$ is trivial for all seminormed local coefficient systems on $W$ (Corollary~\ref{cor:weak invariance bc}). We pull back~$\calA$ to a~$\mathfrak{X}$-coefficient system on $W$ along $f$, and then further to a $\mathfrak{X}$-coefficient system $(f^*\calA)_y$ on the fiber~$f^{-1}(y) = K(\Gamma,n-1)$ for some $y \in K(\Gamma, n)$. We have $H^i_b(K(\Gamma,n-1);(f^*\calA)_y)=0$ for all~$i\ge 1$ by the inductive hypothesis. Then the claim follows by applying Corollary~\ref{cor:acyclic fibre}. 
	\end{proof}
\end{lemma}

\begin{lemma}
\label{lem:acyclic fiber}
    Let~$n\in \N$ and let~$X$ be a connected (based) Kan complex with~$\pi_i(X)\in \frakX$-$\BAc$ for all~$i\in \{1,\ldots,n\}$.
    Then we have $H^i_b(X;\calA)=0$ for all~$i\in \{1,\ldots,n\}$ and all $\frakX$-coefficient systems~$\calA$ on~$X$.
    \begin{proof} We consider a part of the Postnikov tower of $X$~\eqref{eqn:Postnikov}
        \[
            X\to P_n(X)\to P_{n-1}(X)\to \cdots\to P_2(X)\to P_1(X).
        \]
        Here $P_1(X)$ is a model for~$K(\pi_1(X),1)$ and the fiber of the map $P_2(X)\to P_1(X)$ is a model for~$K(\pi_2(X),2)$.
        It follows from Lemma~\ref{lem:Eilenberg MacLane} and the Serre spectral sequence in bounded cohomology (Theorem~\ref{thm:sss}) that we have $H^i_b(P_2(X);\calA)=0$ for all~$i\ge 1$ and all $\frakX$-coefficient systems~$\calA$.
        (Since every map in the tower induces a canonical isomorphism on~$\pi_1$, we identify the local coefficient systems that arise from pulling back.)
        Inductively, we obtain that $H^i_b(P_n(X);\calA)=0$ for all~$i\ge 1$ and all $\frakX$-coefficient systems~$\calA$.
        Since the map $X\to P_n(X)$ is an $(n+1)$-equivalence, the claim follows from Corollary~\ref{cor:n-invariance bc}. 
    \end{proof}
\end{lemma}

\begin{example}
    Since for $n \geq 2$ the (simplicial) homotopy groups are abelian~\cite[Corollary~I.7.7]{goerss2009simplicial}, therefore amenable, Lemma~\ref{lem:acyclic fiber} and Example~\ref{ex:classes of coeffs} imply that a connected Kan complex~$X$ with amenable fundamental group has vanishing bounded cohomology $H_b^i(X; \calA) = 0$ for all $i \geq 1$ and all local coefficient systems~$\calA$ of dual Banach modules. Recalling Remark~\ref{rem:bc:spaces:and:bc:groups}, this corresponds to  Gromov's well-known result about the vanishing of bounded cohomology for path-connected topological spaces with amenable fundamental group.
\end{example}

Lemma~\ref{lem:acyclic fiber} and Corollary~\ref{cor:acyclic fibre} together yield:

\begin{thm}[Mapping theorem for classes of coefficients]
\label{thm:mapping thm class}
	Let~$f\colon X\to Y$ be a Kan fibration between connected simplicial sets with connected fiber~$F$.
    Let~$\calA$ be a seminormed local coefficient system on~$X$ that restricts to a $\frakX$-coefficient system on~$F$.
	Let~$n\in \N$ and suppose that $\pi_i(F)\in \mathfrak{X}$-$\BAc$ for all~$i\le n$ (for any basepoint).
	Then the canonical map
	\[
		 H^i_b(Y;\calH^0_b(F;\calA))\to H^i_b(X;\calA)
	\]
	is an isomorphism (of $\R$-modules) for all~$i\le n$ and injective for~$i=n+1$.
\end{thm}

\begin{rem}
A converse of Theorem~\ref{thm:mapping thm class} for the category $\frakX$ of dual Banach $\R$-modules also holds (see \cite[Theorem A]{Moraschini-Raptis}). This gives a characterization of the maps between path-connected spaces whose homotopy fiber is path-connected and has amenable fundamental group. In addition, an analogous characterization is known for the class of local coefficient systems $\calA$ of dual Banach $\R$-modules on $X$ that are pulled back from $Y$. This involves the category $\underline{\R}$ (see \cite[Theorem~C]{Moraschini-Raptis}) or, equivalently, the category of dual Banach modules and identity morphisms (see \cite{Moraschini-Raptis2}). It would be interesting to know whether such characterizations can be extended to general classes of local coefficient systems. 
\end{rem}

\begin{proof}[Proof of Theorem~\ref{thm:mapping thm}] 
    This is obtained  from Theorem~\ref{thm:mapping thm class} applied to the classes of coefficients in Example~\ref{ex:classes of coeffs}.
\end{proof}

\subsection{Serre fibrations of topological spaces}
Given a Serre fibration $X\to Y$ of topological spaces, applying Theorem~\ref{thm:sss} to the Kan fibration $\Sing(X)\to \Sing(Y)$ yields:

\begin{thm}
\label{thm:Serre fibrations}
    Let~$f\colon X\to Y$ be a Serre fibration of topological spaces and let~$\calA$ be a seminormed local coefficient system on~$X$.
    Then there is a first-quadrant spectral sequence~$(E^{\bullet,\bullet}_r)_r$ converging to~$H^*_b(X;\calA)$.
    We have an isomorphism of seminormed $R$-modules
    \[
        E^{p,q}_2\cong H^p_b(Y;\calH^q_b(F;\calA))
    \]
    in each of the following cases:
    \begin{enumerate}[label=\enum]
        \item $q=0$;
        \item $Y$ is connected and for some~$y\in Y$ the seminormed $R$-cochain complex $C^\bullet_b(F_y;\calA_y)$ satisfies~$\UBC^q$.
    \end{enumerate}
\end{thm}

We note, however, that an advantage of the Serre spectral sequence in the more general context of Kan fibrations is that this works also for simplicial sets which are not Kan complexes. 

For covering maps, which are fibrations with discrete fibers, the terms~$E^{p,q}_2$ are trivial for~$q\ge 1$, so we conclude:

\begin{cor}
\label{cor:covering}
Let~$f\colon X\to Y$ be a covering map of topological spaces with (discrete) fiber~$F$.
	Let~$\calA$ be a seminormed local coefficient system on~$X$.
	Then the canonical map
	\[
		 H^i_b(Y;\calH^0_b(F;\calA))\to H^i_b(X;\calA)
	\]
	is an isomorphism (of $R$-modules) for all $i \geq 0$.
\end{cor}

\begin{example}
   An inclusion of discrete groups $H\hookrightarrow G$ induces a covering map (up to homotopy) of aspherical spaces $BH\to BG$ with fiber~$G/H$. 
Applying Corollary~\ref{cor:covering} (and recalling Remark~\ref{rem:bc:spaces:and:bc:groups}), we recover the (non-isometric form of the) Eckmann--Shapiro Lemma in bounded cohomology~\cite{monod}. 
\end{example}

The case of covering spaces generalizes as follows -- we leave variations of this result for other classes $\mathfrak{X}$-$\BAc$ to the interested reader. 

\begin{cor}
\label{cor:covering2}
Let~$f\colon X\to Y$ be a Serre fibration of path-connected topological spaces with fiber~$F$. Suppose that the homotopy groups of each path-component of~$F$ are amenable and let~$\calA$ be a seminormed local coefficient system on~$X$ of dual Banach $\R$-modules. Then the canonical map $$H^i_b(Y;\calH^0_b(F;\calA)) \to H^i_b(X;\calA)$$ is an isomorphism (of $\R$-modules)
for all $i \geq 0$.
\end{cor}
\begin{proof}
By taking path-components fiberwise, we obtain an associated covering map $f' \colon X' \to Y$ with fiber $\pi_0(F)$ and a map $p \colon X \to X'$ over $Y$. (This is the Moore--Postnikov truncation of the map $f$.) The space $X'$ is again path-connected and the map $p$ is $\pi_1$-surjective with amenable kernel and path-connected homotopy fiber. Specifically, the homotopy fiber of $p$ is identified with a path-component $F'$ of $F$.

Thus, using Theorem \ref{thm:mapping thm class} for the class of dual Banach $\R$-modules, we obtain canonical isomorphisms (of $\R$-modules) for all $i \geq 0$:
$$H^i_b(X';\calH^0_b(F';\calA)) \cong H^i_b(X;\calA).$$
We write $\calA' \coloneqq\calH^0_b(F';\calA)$ for brevity. 
On the other hand, the spectral sequence for the covering map $f'$ yields canonical isomorphisms (of $\R$-modules):
$$H^i_b(Y;\calH^0_b(\pi_0(F);\calA')) \cong H^i_b(X';\calA')$$
for all $i \geq 0$. 
Lastly, by Theorem~\ref{thm:mapping thm class} applied to the family of Serre fibrations $\{F_y \to \pi_0(F_y)\}_{y \in Y}$, the local coefficient system $\calH^0_b(\pi_0(F);\calA')$
on $Y$ is identified canonically with the local coefficient system 
$\calH^0_b(F;\calA)$ on $Y$. 
\end{proof}

\subsection{Group extensions}
The Serre spectral sequence for Kan fibrations recovers the Lyndon--Hochschild--Serre spectral sequence for group extensions~\cite{Noskov92, Burger_Monod_2002, monod} (see also~\cite{echtler}).
Given a group extension $1\to \Lambda\to \Gamma\to Q\to 1$, applying Theorem~\ref{thm:sss} to the Kan fibration $N(\Gamma)\to N(Q)$ yields:

\begin{thm}[Lyndon--Hochschild--Serre spectral sequence in bounded cohomology]
\label{thm:LHS}
    Let $1\to \Lambda\to \Gamma\to Q\to 1$ be a group extension.
    Let~$A$ be a seminormed $\Gamma$-module.
    Then there is a first-quadrant spectral sequence~$(E^{\bullet,\bullet}_r)_r$ converging to~$H^*_b(\Gamma;A)$.
    We have an isomorphism of seminormed $R$-modules
    \[
        E^{p,q}_2\cong H^p_b(Q;H^q_b(\Lambda;A))
    \]
    in each of the following cases:
    \begin{enumerate}[label=\enum]
        \item the group~$Q$ is finite;
        \item $q=0$;
        \item $A$ is a Banach $\Gamma$-module and~$H^q_b(\Lambda;A)$ is Banach.
    \end{enumerate}
\end{thm}

\begin{rem}
It is claimed in the literature (see~\cite[Proposition~4.2.2]{Burger_Monod_2002},\cite[Proposition~12.2.2]{monod}, \cite{echtler}) that there is an identification $E^{0,q}_2\cong H^0_b(Q;H^q_b(\Lambda;A))$. However, the proof of this claim seems to be incomplete and still requires the additional assumption that $H^q_b(\Lambda;A)$ is Banach. 
See also Remark~\ref{rem:non-BBC}.
\end{rem}

\subsection{(Co-)amenability}
Given a group extension $1 \to \Lambda \xrightarrow{i} \Gamma \to Q \to 1$, where~$Q$ is amenable, a transfer argument shows that the induced map $H_b^*(i)$ is injective in all degrees~\cite{Monod:Popa}. We deduce the corresponding statement for spaces from the Serre spectral sequence (Theorem~\ref{thm:sss}). 

\begin{defi}
Let~$\Gamma$ be a group. We say that a $\Gamma$-set~$S$ is \emph{amenable} if it admits a $\Gamma$-invariant mean, that is,
a $\Gamma$-invariant 
map $\mu \colon \ell^{\infty}(S) \to \R$ of Banach $\Gamma$-modules satisfying 
$\mu(1_S) = 1$ and $\mu(f)\ge 0$ for~$f\ge 0$.
\end{defi}

Note that any $\Gamma$-invariant map $\mu\colon \ell^\infty(S)\to \R$ of Banach $\Gamma$-modules satisfying $\mu(1_S)\neq 0$ gives rise to a $\Gamma$-invariant mean by taking an appropriate absolute value of~$\mu$ and normalizing (see~\cite[Proof of Proposition~3]{Monod:Popa}). In particular, a $\Gamma$-set~$S$ is amenable if and only if the Banach $\Gamma$-submodule 
$\R$ of $\ell^{\infty}(S)$ (given by the constant functions) is a retract. Similarly, we have:

\begin{rem}
Let $S$ be an amenable $\Gamma$-set and let~$A$ be a dual Banach $\Gamma$-module. Then a $\Gamma$-invariant mean $\mu \colon \ell^\infty(S) \to \R$ induces a map of Banach $\Gamma$-modules $\ell^{\infty}(S, A) \to A$ which has the inclusion of constant functions $A \to \ell^{\infty}(S, A)$ as a section.  
\end{rem}

\begin{example}
Given a normal subgroup $H$ of $\Gamma$, the $\Gamma$-set $\Gamma/H$ is amenable if and only if the group $\Gamma/H$ is an amenable group. 
In this case, the subgroup $H$ is called \emph{co-amenable} in $\Gamma$.  
\end{example}

\begin{prop}
\label{prop:co-amenability}
Let $F \xrightarrow{i} E \xrightarrow{p} B$ be a Serre fibration of path-connected topological spaces, where $\pi_1(B)$ is amenable. Let $\calA$ be a seminormed local coefficient system of dual Banach modules on $E$. Then the induced map
$$H^*_b(i)\colon H^*_b(E; \calA) \to H^*_b(F; i^*\calA)$$
 is split-injective. 
\end{prop}
\begin{proof} We consider the spectral sequence in bounded cohomology for the induced homotopy fiber sequence $\Omega B \xrightarrow{j} F \xrightarrow{i} E$ (after a choice of basepoint in $B$). The path-components of the loop space $\Omega B$ have amenable (in fact abelian) fundamental groups. 
Then Corollary~\ref{cor:covering2} (applied to the associated Serre fibration) yields canonical isomorphisms:
$$H^*_b(E; \calH^0_b(\Omega B; j^*i^*\calA)) \cong H^*_b(F; i^*\calA).$$
The comparison map to the homotopy fiber sequence $\ast \xrightarrow{e} E \to E$ yields a commutative square:
\[\begin{tikzcd}
    H^*_b(E; \calH^0_b(\Omega B; j^*i^*\calA)) \ar{r}{\cong} & H^*_b(F; i^*\calA)  \\
H^*_b(E; \calH^0_b(\ast; e^*\calA)) \ar{u} \ar{r}{\cong} & H^*_b(E; \calA) \ar{u}[swap]{H^*_b(i)} 
\end{tikzcd}\]
The coefficient module $\calH^0_b(\Omega B; j^*i^*\calA)$ can be identified with the Banach $\pi_1(E,e)$-module 
$\ell^{\infty}(\pi_1(B,b), \calA(e))$, where $p(e)=b$ and $\pi_1(E, e)$ acts on $\pi_1(B, b)$ and on~$\calA(e)$ in the obvious way. Moreover, the left map above is induced by the canonical inclusion $\calA(e) \to \ell^{\infty}(\pi_1(B, b), \calA(e))$. Since $\pi_1(B, b)$ is amenable (as $\pi_1(E, e)$-set) and $\calA(e)$ is a dual Banach module, it follows that there is a retraction (as $\pi_1(E, e)$-modules)
$$\ell^{\infty}(\pi_1(B, b), \calA(e)) \to \calA(e).$$
As a consequence, the left map admits a left inverse, and therefore, so does $H^*_b(i)$, as required. 
\end{proof}

\subsection{Simplicial volume of manifold bundles}
\label{sec:sv}
Simplicial volume is a homotopy invariant of oriented closed connected manifolds introduced by Gromov~\cite{vbc}. 

\begin{defi}[Simplicial volume]
Let $M$ be an oriented closed connected $n$-manifold and let $[M] \in H_n(M; \R)$ denote the real fundamental class of $M$. The \emph{simplicial volume of~$M$} is 
\[
\|M\| \coloneqq \inf \Big\{\sum_{i = 1}^k |a_i| \, \Big| \, c=\sum_{i=1}^k a_i \sigma_i\in C_n(M;\R), [c] = [M] \Big\} \in [0, +\infty).
\]
\end{defi}

By the duality principle~\cite{vbc, Frigerio:book}, the simplicial volume~$\|M\|$ is positive if and only if the comparison map $\comp^n\colon H_b^n(M; \R) \to H^n(M; \R)$ in degree~$n$ is surjective.
In particular, the simplicial volume of manifolds with amenable fundamental is zero.
The precise value of the simplicial volume (when non-zero) is known only in a few cases, e.g., for hyperbolic manifolds~\cite{Thurston, vbc} and products of hyperbolic surfaces~\cite{bucher2008simplicial}.

We consider the simplicial volume of manifolds arising as the total space of a fiber bundle $F^{n-d}\to E^n\to B^{d}$ of oriented closed connected manifolds of the respective dimensions.
For the trivial bundle $F\times B$, a classical computation~\cite{vbc, Bucher:fiber} yields the inequalities
\begin{equation}
\label{eqn:sv product}
    \|F\|\cdot \|B\|\le \|F\times B\|\le \binom{n}{d}\cdot \|F\|\cdot \|B\|.
\end{equation}
In particular, the right inequality in~\eqref{eqn:sv product} shows: if~$\|F\|=0$ or~$\|B\|=0$, then $\|F\times B\|=0$.
This implication is not true in general for (non-trivial) fiber bundles, e.g., witnessed by hyperbolic 3-manifolds arising as mapping tori of surfaces.
We show that the implication holds under additional assumptions:

\begin{prop}
\label{prop:sv vanishing}
    Let $f\colon E^n\to B^d$ be a fiber bundle of oriented closed connected manifolds with fiber~$F$.
    Suppose that~$H^i_b(F;\R)=0$ for all~$i\in \{n-d+1,\ldots,n\}$.
    If~$\|F\|=0$ or~$\|B\|=0$, then~$\|E\|=0$.
    \begin{proof}
        Since~$H^i_b(F;\R)=0$ for all~$i\in \{n-d+1,\ldots,n\}$, the entries~$E_2^{p,q}$ of the Serre spectral sequence for bounded cohomology are trivial for all~$q\in \{n-d+1,\ldots,n\}$.
        It follows that there are maps
        \[
            H^n_b(E;\R)\onto E_\infty^{d,n-d}\into E_2^{d,n-d}.
        \]
        The comparison map of spectral sequences (Proposition~\ref{prop:comparison ss}) yields a commutative diagram
        \[\begin{tikzcd}
            H^n_b(E;\R)\ar{r}\ar{dd} & E_2^{d,n-d}\ar{d} \\
            & H^d_b(B;\calH^{n-d}_b(F;\R))\ar{d} \\
            H^n(E;\R)\ar{r}{\cong} & H^d(B;\calH^{n-d}(F;\R))
        \end{tikzcd}\]
        The map $H^d_b(B;\calH^{n-d}_b(F;\R))\to H^d(B;\calH^{n-d}(F;\R))$ factorizes as follows:
        \[\begin{tikzcd}
            H^d_b(B;\calH^{n-d}_b(F;\R))\ar{r}\ar{d} & H^d_b(B;\calH^{n-d}(F;\R)) \cong H^d_b(B ; \R) \ar{d}
            \\
            H^d(B;\calH^{n-d}_b(F;\R))\ar{r} & H^d(B;\calH^{n-d}(F;\R)) \cong H^d(B; \R).
        \end{tikzcd}\]
        If $\|F\|=0$, then the top map is trivial; if $\|B\|=0$, then the right map is trivial. In both cases, the map $H^d_b(B;\calH^{n-d}_b(F;\R))\to H^d(B;\calH^{n-d}(F;\R))$ is trivial, and hence the comparison map $H^n_b(E;\R)\to H^n(E;\R)$ is trivial.
    \end{proof}
\end{prop}

The right inequality in~\eqref{eqn:sv product} shows: if~$\|F\|>0$ and~$\|B\|>0$, then~$\|F\times B\|>0$.
It seems to be unknown whether this implication holds in general for (non-trivial) fiber bundles.
We show that the implication holds under stronger assumptions:

\begin{prop}
\label{prop:sv positive}
    Let $f\colon E^n\to B^d$ be a fiber bundle of oriented closed connected manifolds with fiber~$F$.
    Suppose that the following hold:
    \begin{enumerate}[label=\enum]
        \item $C^\bullet_b(F;\R)$ satisfies~$\UBC^q$ for all~$q\in \{0,1,\ldots,n-d\}$;
        \item $H^p_b(B;\calH^q_b(F;\R))=0$ for all~$p,q$ with $p+q\in \{n,n+1\}$ and~$q\le n-d-1$;
        \item the map $H^d_b(B;\calH^{n-d}_b(F;\R))\to H^d(B;\calH^{n-d}(F;\R))$ is non-trivial.
    \end{enumerate}
    Then we have $\|E\|>0$.
    \begin{proof}
        Under the given assumptions, the entries~$E_2^{p,q}$ of the Serre spectral sequence for bounded cohomology (Theorem~\ref{thm:sss}) are trivial for all~$p,q$ with $p+q\in \{n,n+1\}$ and $q\le n-d-1$.
        It follows that there are maps
        \[
            E_2^{d,n-d}\onto E_\infty^{d,n-d}\into H^n_b(E;\R).
        \]
        Moreover, we may identify $E^{d,n-d}_2\cong H^d_b(B;\calH^{n-d}_b(F;\R))$.
        The comparison map of spectral sequences (Proposition~\ref{prop:comparison ss}) yields a commutative diagram
        \[\begin{tikzcd}
            H^d_b(B;\calH^{n-d}_b(F;\R))\ar{r}\ar{d} & H^n_b(E;\R)\ar{d} \\
            H^d(B;\calH^{n-d}(F;\R))\ar{r}{\cong} & H^n(E;\R) 
        \end{tikzcd}\]
        By assumption, the left vertical map is non-trivial and hence the right vertical map is non-trivial.
    \end{proof}
\end{prop}

\medskip
\noindent \textbf{Acknowledgments.} 
We thank Clara L\"oh for helpful discussions.

The first and third authors were supported by the SFB~1085 - \emph{Higher Invariants} (Universit\"at Regensburg) funded by the DFG.
The second author was supported by the ERC ``Definable Algebraic Topology" DAT - Grant Agreement no.~101077154. 
This work has been funded by the European Union - NextGenerationEU under the National Recovery and Resilience Plan (PNRR) - Mission 4 Education and research - Component 2 From research to business - Investment 1.1 Notice Prin 2022 -  DD N.~104 del 2/2/2022, from title ``Geometry and topology of manifolds", proposal code 2022NMPLT8 - CUP J53D23003820001.


\bibliographystyle{alpha}
\bibliography{svbibnew}

\setlength{\parindent}{0cm}

\end{document}